\documentclass{amsart}
\usepackage[english]{babel}
\usepackage{amsfonts,enumerate,amssymb,latexsym,color,tcolorbox,graphicx}
\usepackage{amsthm}
\usepackage{amsmath}
\usepackage{bm}
\usepackage{mathrsfs}
\usepackage[colorlinks,
           linkcolor=black,
            anchorcolor=black,
           citecolor=black]{hyperref}

\graphicspath{{figures/}}

\theoremstyle{plain}
\newtheorem{theorem}{Theorem}[section]
\newtheorem{lemma}[theorem]{Lemma}
\newtheorem{coro}[theorem]{Corollary}
\newtheorem{prop}[theorem]{Proposition}

\newtheorem{fact}{Fact}

\theoremstyle{definition}
\newtheorem{de}[theorem]{Definition}
\newtheorem{exa}[theorem]{Example}
\newtheorem{question}[theorem]{Question}

\theoremstyle{remark}
\newtheorem{remark}[theorem]{Remark}

\numberwithin{equation}{section}

\def\R{\mathbb R}
\def\Z{\mathbb Z}
\def\D{\mathcal D}
\def\A{\mathcal A}
\def\C{\mathscr C}
\def\c{\mathcal C}
\def\B{\mathcal B}

\begin{document}

\title[Fractal squares with finitely many connected components]{Fractal squares with finitely many \\ connected components}

\author{Jian-Ci Xiao}
\address{School of Mathematical Sciences, Zhejiang University, Hangzhou 310027,
China}\thanks{This work is partly supported by NSFC grant 11771391.}
\email{jcxshaw24@zju.edu.cn}



\date{}



\begin{abstract}

In this paper, we present an effective method to characterize completely when a disconnected fractal square has  only finitely many connected components. Our method is to establish some graph structures on fractal squares to reveal the evolution of the connectedness during their geometric iterated construction. We also prove that every fractal square contains either finitely or uncountably many connected components. A few examples, including the construction of fractal squares with exactly $m \geqslant 2$ connected components, are added in addition.
\end{abstract}

\maketitle

\section{Introduction}
\label{intro}
\setcounter{equation}{0}

\subsection{Background and some notations}
\text{ }

The topological properties of self-similar sets have been studied frequently in recent years. One can see the equivalence of connectedness and path-connectedness in Hata's classic article~\cite{Hata} (or see Kigami~\cite{Kigami}). In~\cite{LuoRaoTan}, Luo, Rao, and Tan studied the topological properties of the interior and boundary of self-similar sets satisfying the open set condition. Another interesting topic is the Lipschitz equivalence between self-similar sets, for which one can refer to Falconer and Marsh~\cite{FM92}, Rao-Ruan-Wang~\cite{RRW12,RRW13} etc. Research on other aspects can be seen in Bandt and Keller~\cite{BanKe}, Luo and Wang~\cite{LuoWang}, Roinestad~\cite{RoinM, RoinP} etc.

For a fixed integer $N \geqslant 2$ and a non-empty digital set $\D \subset \{0,1,\ldots,N-1\}^2$, there exists a unique non-empty compact set $F = F(N, \D)$ satisfying the set equation (see \cite{Fal14, Hut81})
\begin{equation}\label{eq:fracsq}
F = \bigcup_{d \in \D} \varphi_d(F),
\end{equation}
where $\varphi_d(x)=(x+d)/N$. In other words, $F$ is the self-similar attractor of the iterated function system $\{\varphi_d:d \in \D\}$. Note that the condition \eqref{eq:fracsq} is equivalent to $F=(F+\D)/N$ and we always call $F$ a \emph{fractal square}. One can also consider fractal squares through the following geometric iterated construction: let $Q_0 = [0, 1]^2$ and recursively define
\begin{equation}\label{eq:qn} 
Q_{n+1} = \bigcup_{d \in \D}\varphi_d(Q_n) = \frac{Q_n+\D}{N}, \quad n=0,1,2,\ldots, 
\end{equation}
then $Q_{n+1} \subset Q_n$ for each $n$, and $F=\bigcap_{n=1}^\infty Q_n$. A classic example of fractal squares is the Sierpi\'nski carpet (one can of course regard fractal squares as a generalization of it). In~\cite{LauLuoRao}, Lau, Luo and Rao provided a characterization of the topological structure of fractal squares through their connected components. They claim that if $F$ is not totally disconnected, then either it contains a non-trivial connected component which is not a line segment, or all non-trivial connected components of $F$ are parallel line segments. Some other properties such as the Lipschitz equivalence, gap sequences, cut points and cut index of fractal squares have also been studied in \cite{LMR19, RuanWang}.


Naturally, we have the following question:

\begin{question}
  Is it possible for a disconnected fractal square to contain only a finite number of connected components? And if this were the case, can we present a method to characterize when this happens and to determine the number of connected components?
\end{question}

In this paper, we will focus on these problems. In fact, we obtain an affirmative answer to the first one and answer the second one by presenting a complete characterization of fractal squares with finitely many connected components.  Our method is to construct a graph $G_F$ (and another graph $G'_F$ if necessary) corresponding to $F$ and study their relations on connectedness. By the way, in Cristea and Steinsky~\cite{CriSte} the authors constructed a graph similar to our first one and gave a method to determine whether a fractal square is connected or not.

We list below two important notations used throughout this paper.
\begin{itemize}
\item For any set $A$, let $\C(A)$ be the collection of all connected components of $A$.
\item For any set (or collection) $A$, the cardinality of $A$ is denoted by $\#A$.
\end{itemize}

\subsection{Construction of graphs and statement of results}
\text{ }

Suppose $F=F(N,\D)$ is a fractal square. We first introduce a concept of ``connectedness'' in $\D$.

\begin{de}
A set $\A \subset \D$ is said to be \emph{connected} if for any $d, d' \in \A$, there exist $d_1, \ldots, d_n \in \A$ such that $d_1=d, d_n=d'$, and $\varphi_{d_i}(F) \cap \varphi_{d_{i+1}}(F) \neq \varnothing$ for each $1 \leqslant i \leqslant n-1$.
The maximal (ordered by inclusion) connected subsets of $\D$ are called \emph{connected components} of $\D$ (so the notation $\C(\D)$ makes sense). 
\end{de}

We should point out that $F$ is connected if and only if $\D$ is connected, which is a classic result (for example, see~\cite{Hata, Kigami}).
Since disconnected fractal squares are of our major concern, we always assume that $\#\C(\D)=m \geqslant 2$ and set $\C(\D)=\{\D_1,\ldots,\D_m\}$. 

For $1\leqslant i \leqslant m$, we denote
\begin{equation}\label{eq:Fi}
F_i = \bigcup_{d \in \D_i} \varphi_d(F).
\end{equation} 
In other words, $F_i$ is the part of $F$ lying in $\bigcup_{d \in \D_i}\varphi_d([0,1]^2)$ (a union of squares with side length $1/N$). By the definition of the connectedness in $\D$ one can easily see that $F_i \cap F_j = \varnothing$ if $i \neq j$ and $F = \bigcup_{i=1}^m F_i$. 
Note that $\varphi_d(F_i)$ is a scaled copy of $F_i$ for each $d \in \D$, and clearly $\varphi_d(F)=\bigcup_{i=1}^m \varphi_d(F_i)$.
	
\noindent\textbf{Construction of the graph $G_F$}. The vertex set of $G_F$ is $\{(d,i): d \in \D, 1\leqslant i \leqslant m\}$, and there exists an edge joining $(d_1,i_1)$ and $(d_2,i_2)$ if and only if $\varphi_{d_1}(F_{i_1}) \cap \varphi_{d_2}(F_{i_2}) \neq \varnothing$. The graph $G_F$ is called \emph{the level-1 graph of $F$} (see Example~\ref{exa:1} for an illustration). 

\begin{de}
Suppose $G$ is a graph with vertex set $V$ and edge set $E$. We call $V' \subset V$ \emph{connected} if every pair of distinct vertices in $V'$ can be joined by a path in $E$. Connected components of $G$ are maximal (ordered by inclusion) connected subsets of $V$.
\end{de}
We have to mention here that our definition of connected components of a graph is slightly different from the customary one (maximal connected subgraph). This is settled mainly for the notational convenience of later exposition.

It is easy to see that $\#\C(F)\geqslant \#\C(G_F) \geqslant \#\C(\D)$ (Lemma~\ref{lem:2-2}). The following theorem presents a sufficient condition for a fractal square to have finitely many connected components.

\begin{theorem}\label{thm:main1}
Suppose $\#\C(\D) \geqslant 2$. Then $\#\C(F)=\#\C(\D)$ if and only if $\#\C(G_F)=\#\C(\D)$.
\end{theorem}

In general, $G_F$ may contain more connected components than $\D$ (e.g., see Example~\ref{exa:last}). In this case, we know that $\#\C(F)>\#\C(\D)$ by Theorem~\ref{thm:main1}. It is of interest that whether $F$ can still have only finitely many connected components or not. Toward this end, we turn to the construction of the so called ``level-2" graph, which can be regarded as an advanced version of the previous one. Suppose $\#\C(G_F)=M$, say $\C(G_F) = \{\c_1, \c_2, \cdots, \c_M\}$. 

\begin{lemma}\label{lem:3}
If we let $\D^* = N\D+\D$ and $F^*(=F)$ be the fractal square satisfying $F^*=\bigcup_{d' \in \D^*}\psi_{d'}(F^*)$, where $\psi_{d'}(x)=(x+d')/N^2$, then $\#\C(\D^*)=\#\C(G_F)=M$ and rearranging if necessary we have
\begin{equation}\label{eq:ej} 
\bigcup_{(d,i) \in \c_j} \varphi_d(F_i) = \bigcup_{d' \in \D^*_j} \psi_{d'}(F^*) =: F^*_j, \quad 1 \leqslant j \leqslant M, 
\end{equation}
where $\D^*_1, \ldots, \D^*_M$ are connected components of $\D^*$.
\end{lemma}

Notice that $F^*$ is obtained by a geometric iterated construction different from the one generating $F$, although as the limit set $F^*$ coincide with $F$. This leads to the adoption of above new notations such as $\D^*, \D^*_j, F^*, F^*_j$, etc. 


\noindent\textbf{Construction of the graph $G'_F$}. The vertex set of $G'_F$ is $\{\langle d,j \rangle: d \in \D, 1\leqslant j \leqslant M\}$. Here we use $\langle\cdot,\cdot\rangle$ only to distinguish it from the notation $(\cdot,\cdot)$ in the construction of $G_F$. Further, there exists an edge joining $\langle d_1, j_1 \rangle$ and $\langle d_2, j_2 \rangle$ if and only if $\varphi_{d_1}(F^*_{j_1}) \cap \varphi_{d_2}(F^*_{j_2}) \neq \varnothing$. The graph $G'_F$ is called \emph{the level-2 graph of $F$} (see again Example~\ref{exa:last} for an illustration).

With the aid of graphs $G_F$ and $G'_F$, we can present a complete characterization of fractal squares with finitely many connected components.

\begin{theorem}\label{thm:main2}
  A disconnected fractal square $F$ has finitely many connected components if and only if $\#\C(G_F)=\#\C(G'_F)$. Further, in the case that $\#\C(G_F)=\#\C(G'_F)$, $\#\C(F)$ equals this common value.
\end{theorem}

As for the cardinality of $\C(F)$, we have the following result.
\begin{theorem}\label{thm:main3}
For any fractal square $F$, $\C(F)$ is either a finite or an uncountable set.	
\end{theorem}

This paper is organized as follows. In Section 2 we prove Theorem~\ref{thm:main1} and construct a class of fractal squares with exactly $m \geqslant 2$ connected components. In Section 3 we obtain deeper information on the level-1 graph. In Section 4 we prove Theorem~\ref{thm:main2}. The proof of Theorem~\ref{thm:main3} and some further remarks are added in Section 5.

\section{Proof of Theorem \ref{thm:main1}}

We start with giving an example of the level-1 graph.

\begin{exa}\label{exa:1}
Let $F = (F+\D)/5$ be the fractal square shown in Figure~\ref{exa:1f1}.
\begin{figure}[htbp]
\centering\includegraphics[scale=0.4]{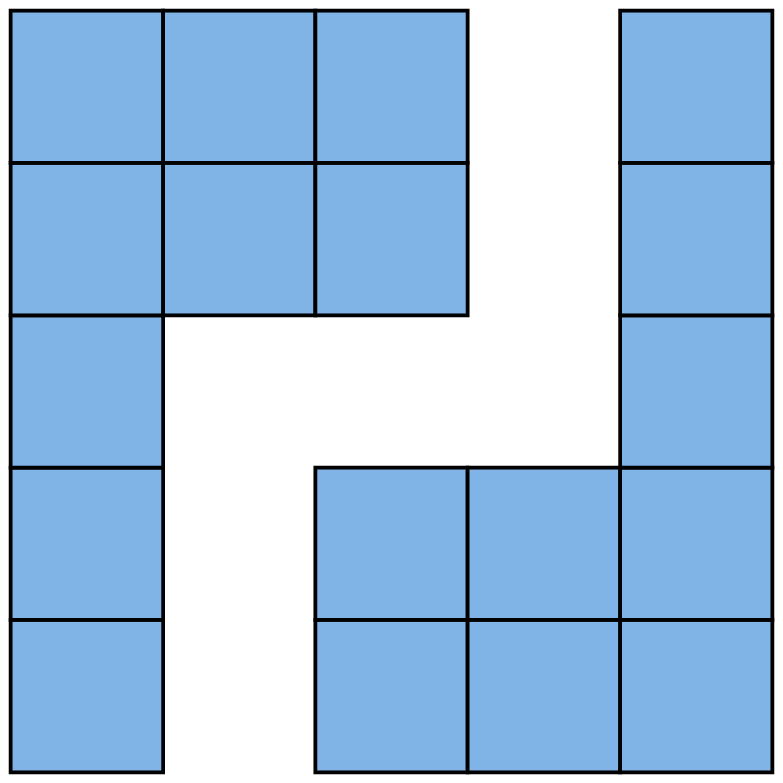} \quad \includegraphics[scale=0.4]{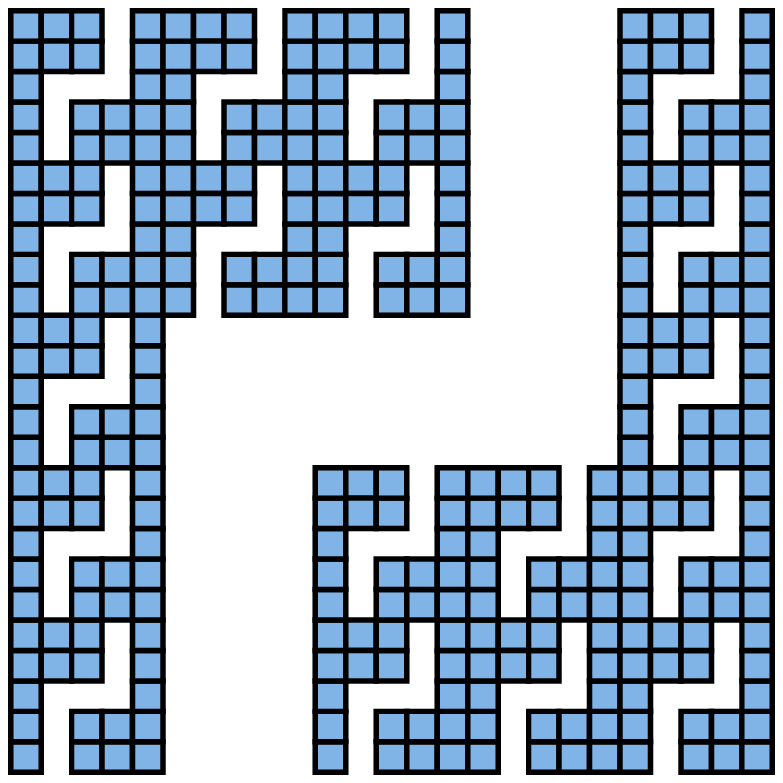} \quad \includegraphics[scale=0.38]{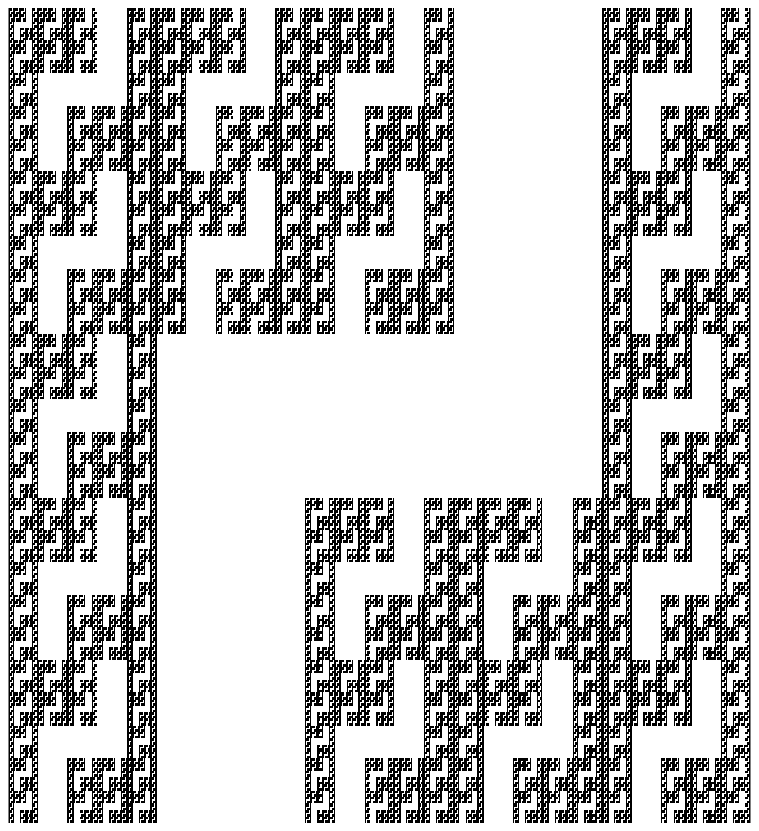}
\vspace{-0.5cm}
\caption{$Q_1, Q_2$ and $F$, where $\#\C(G_F)=\#\C(\D)=2$.}\label{exa:1f1}
\end{figure}
Note that here $\#\C(\D)=2$. Let $\D_1 = \{(0, i): 0 \leqslant i \leqslant 4\} \cup \{(1,3), (1,4), (2,3), (2,4)\}$
 and $\D_2=\D \setminus \D_1$. By definition, we can draw its level-1 graph $G_F$ as in Figure~\ref{fig:exa1G_F}.
\begin{figure}[htbp]
\centering\includegraphics[scale=0.43]{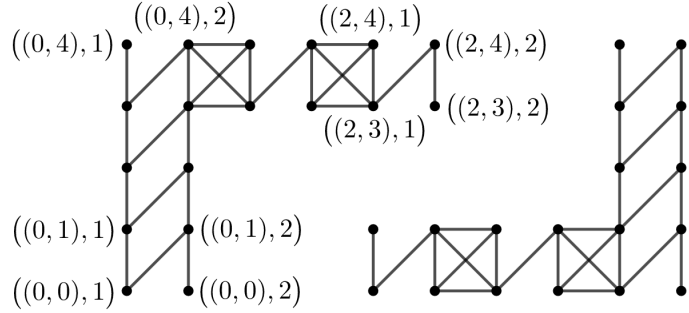}
\vspace{-1cm}
\caption{The level-1 graph of $F$ in Eaxmple~\ref{exa:1}.}\label{fig:exa1G_F}
\end{figure}
\end{exa}


The following observation is straightforward.
\begin{lemma}\label{lem:2-2}
  $\#\C(F)\geqslant \#\C(G_F) \geqslant \#\C(\D)$.
\end{lemma}
\begin{proof}
  Notice that if $(d_1, i_1)$ and $(d_2,i_2)$ belong to different connected components of $G_F$, then $x$ and $y$ belong to different connected component of $F$ for all $x\in \varphi_{d_1}(F_{i_1})$ and $y\in \varphi_{d_2}(F_{i_2})$. This implies that $\#\C(F)\geqslant \#\C(G_F)$.

  Moreover, by the definition of the connectedness in $\D$ and our construction of $G_F$, if $d'$ and $d''$ belong to different connected components of $\D$, say $d' \in \D_1$ and $d'' \in \D_2$ for instance, then $\big(\bigcup_{d \in \D_1} \varphi_d(F)\big) \cap \big((\bigcup_{d \in \D_2} \varphi_d(F) \big) = \varnothing$. So $\big( \bigcup_{d \in \D_1}\bigcup_{i=1}^m \varphi_d(F_i)\big) \cap \big( \bigcup_{d \in \D_2}\bigcup_{i=1}^m \varphi_d(F_i)\big) = \varnothing$. It follows that $(d',i')$ and $(d'',i'')$ must lie in different connected components of $G_F$ for every $1 \leqslant i',i'' \leqslant m$. Hence $\#\C(G_F) \geqslant \#\C(\D)$.
\end{proof}

\begin{coro}\label{cor:ejfij}
Suppose $\C(G_F)=\{\c_1,\ldots,\c_M\}$. Then for any $1 \leqslant j \leqslant M$, there exists a unique $1 \leqslant i_j \leqslant m$ such that $F^*_j \subset F_{i_j}$.	
\end{coro}
\begin{proof}
By the proof of Lemma~\ref{lem:2-2}, if $(d',i')$ and $(d'',i'')$ belong to the same connected component of $G_F$, then $d'$ and $d''$ must belong to the same component of $\D$. Combining this with \eqref{eq:ej} we immediately obtain the desired result.
\end{proof}

We can deduce from the proof of Lemma~\ref{lem:2-2} that if $\#\C(G_F) = \#\C(\D) = m$, then $\{(d,i):d\in \D_j, 1\leqslant i\leqslant m\}$ is a connected component of $G_F$ for each $1\leqslant j\leqslant m$. That is, 
\begin{equation}\label{eq:add1}
	\C(G_F) = \big\{ \{(d, i): d \in \D_1, 1\leqslant i\leqslant m \}, \ldots, \{(d, i): d \in \D_m, 1\leqslant i\leqslant m \} \big\}.
\end{equation}
It is also convenient to denote
\begin{equation}\label{eq:sec1-2}
Q_{n+1}|_{\D_i} := \bigcup_{d \in \D_i} \varphi_d(Q_n), \quad 1\leqslant i\leqslant m, \; n=0,1,2,\ldots.
\end{equation}
Since $\D=\bigcup_{i=1}^n\D_i$ we see that $Q_n = \bigcup_{i=1}^m Q_n|_{\D_i}$. Further, $F_i = \bigcap_{n=1}^\infty Q_n|_{\D_i}$ (recall \eqref{eq:Fi}).

The following result is well-known (see \cite[Exercise 11, Section 26]{Mun}).
\begin{lemma}\label{lem:Mun}
Let $\{A_i\}_{i=1}^\infty$ be a collection of compact and connected subsets of $\R^n$. If $A_{n+1} \subset A_n$ for all $n \in \Z^+$, then $\bigcap_{i=1}^\infty A_i$ is also connected.
\end{lemma}


\begin{proof}[Proof of Theorem \ref{thm:main1}]
The ``only if\," part follows directly from Lemma~\ref{lem:2-2}.

Now we prove the ``if\," part. Suppose $\#\C(G_F)=\#\C(\D)=m$. Then we have \eqref{eq:add1}. 
Toward our end, it suffices to show that $F_i$ is connected for all $1 \leqslant i \leqslant m$. Since $F_i = \bigcap_{n=1}^\infty Q_n|_{\D_i}$, by Lemma~\ref{lem:Mun} this is an easily established result as long as $Q_n|_{\D_i}$ is connected for all $n$ and $i$. We shall prove this by induction.

For any $1 \leqslant i \leqslant m$, note that $Q_1|_{\D_i}$ is the union of some squares with side length $1/N$. Then by the connectedness of $\D_i$ we know $Q_1|_{\D_i}$ is path-connected. Suppose $Q_n|_{\D_1}, \ldots, Q_n|_{\D_m}$ are all path-connected sets for some $n \in \Z^+$. Then for any fixed $i$, we first observe that
\begin{equation*} 
Q_{n+1}|_{\D_i} =  \bigcup_{d \in \D_i} \varphi_d(Q_n) = \bigcup_{d \in \D_i}\bigcup_{j=1}^m \varphi_d(Q_n|_{\D_j}), 
\end{equation*}
which is a union of some path-connected sets. The path connectedness of $Q_{n+1}|_{\D_i}$ then follows immediately if we can show that $\varphi_{d'}(Q_n|_{\D_{j'}})$ and $\varphi_{d''}(Q_n|_{\D_{j''}})$ lie in the same path-connected component of $Q_{n+1}$ for any $d', d'' \in \D_i$ and any $j', j'' \in \{1,\ldots,m\}$.

In fact, since $\{(d, j): d \in \D_i, 1\leqslant j \leqslant m\}$ is connected, there exists a sequence $\{(d_k,j_k)\}_{k=1}^p \subset \D_i \times \{1,\ldots,m\}$ such that $(d_1, j_1) = (d', j')$, $(d_p, j_p) = (d'', j'')$, and
\begin{equation*} 
\varphi_{d_k}(F_{j_k}) \cap \varphi_{d_{k+1}}(F_{j_{k+1}}) \neq \varnothing, \quad k=1,2,\ldots,p-1. 
\end{equation*}
Since $F_{j_k} \subset Q_n|_{\D_{j_k}}$ for all $n$ and $k$, we have
\begin{equation*} 
\varphi_{d_k}(Q_n|_{\D_{j_k}}) \cap \varphi_{d_{k+1}}(Q_n|_{\D_{j_{k+1}}}) \neq \varnothing, \quad k=1,2,\ldots,p-1.
\end{equation*}
By our hypothesis, $Q_n|_{\D_{j_k}}$ is path-connected for all $k$. Since the union of two path-connected sets is also path-connected if one intersects another, we see that $\varphi_{d_1}(Q_n|_{\D_{j_1}})$ ($=\varphi_{d'}(Q_n|_{\D_{j'}})$) and $\varphi_{d_2}(Q_n|_{\D_{j_p}})$ ($=\varphi_{d''}(Q_n|_{\D_{j''}})$) lie in the same path-connected component of $Q_{n+1}$. This is all we need.
\end{proof}



By Theorem~\ref{thm:main1}, the fractal square in Example~\ref{exa:1} has exactly two connected components since in that case we have $\#\C(G_F) = \#\C(\D) = 2$. Now we turn to the construction of fractal squares with exactly $m\geqslant 3$ connected components.

\begin{exa}
Let us start with $m=3$ and $m=4$. In Figure~\ref{fig:exa2} we show the first stage in the geometric construction of two fractal squares (i.e., $Q_1$) respectively.
\begin{figure}[htbp]
\centering \includegraphics[scale=0.4]{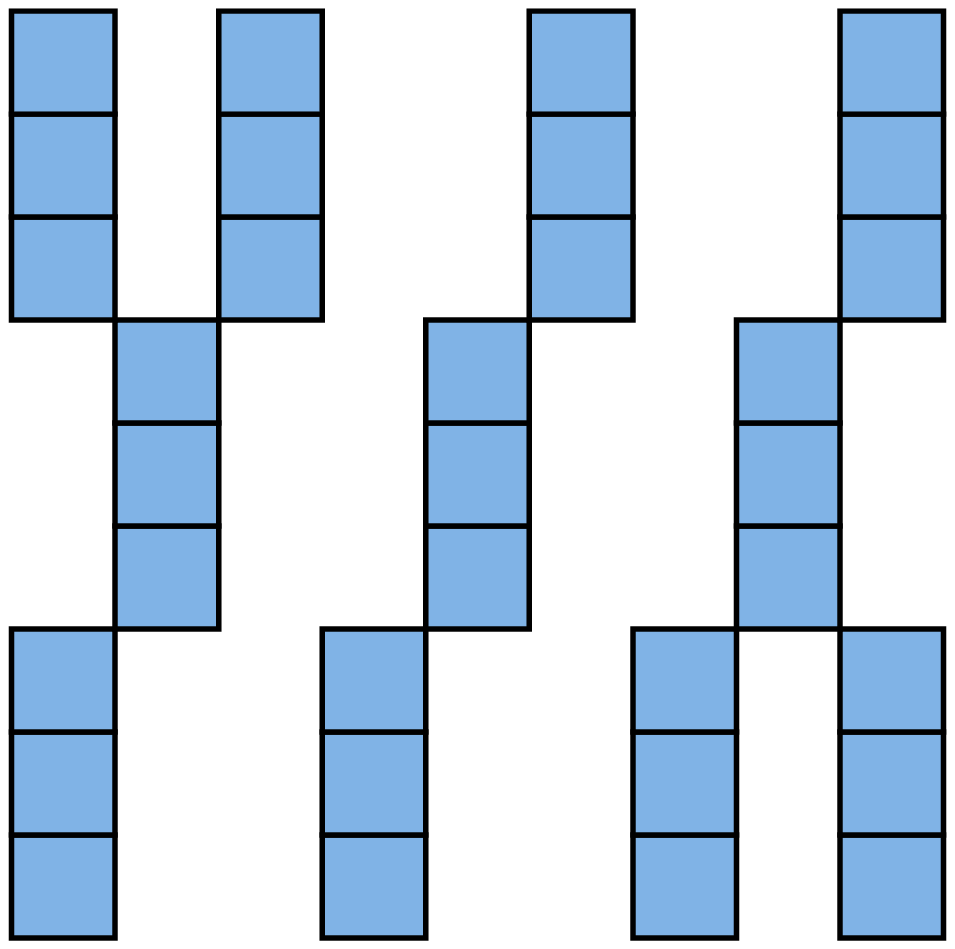} \quad\quad \includegraphics[scale=0.4]{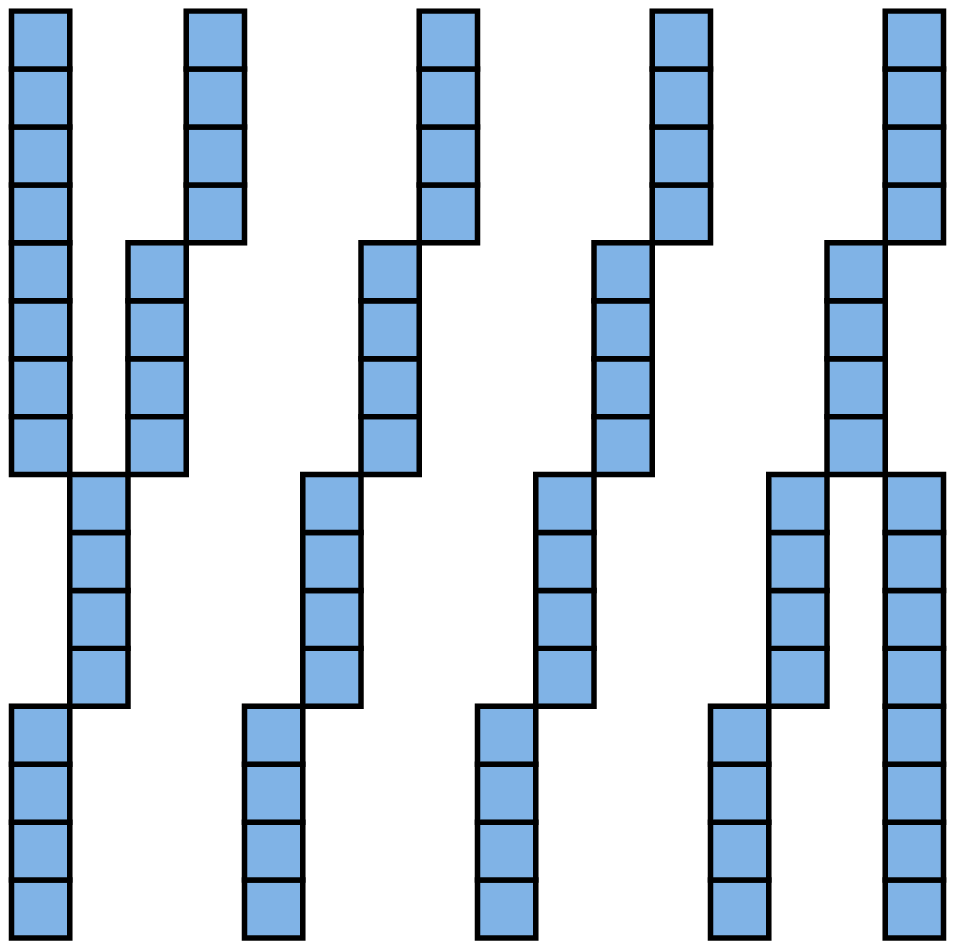} 
\caption{The initial pattern of fractal squares with $\#\C(F)=3$ or $\#\C(F)=4$.}\label{fig:exa2}
\end{figure}
By drawing level-1 graphs we can see that in the former case, $\#\C(G_F)=3$ so that $\#\C(F)=3$, and in the latter case, $\#\C(G_F)=\#\C(F)=4$. We can obtain a fractal square with exactly $m$ connected components for every $m \geqslant 5$ in a similar way. More precisely, first denote
\begin{equation*} 
A := \bigcup_{i=0}^{m-1}\bigcup_{j=im}^{(i+1)m-1} (i, j), \quad B := \bigcup_{j=2m}^{m^2-1} \{(0, j), (m^2-1, m^2-1-j)\}. 
\end{equation*}
Let $\D = B \cup \big( \bigcup_{k=0}^{m-1} (A+(km, 0)) \big)$, then $F=F(m^2,\D)$ is a fractal square with $\#\C(F)=m$ as desired.
\end{exa}

\section{Further study on the level-1 graph}

In this section we try to go further: \textit{Can a fractal square contain more but still finitely many connected components than its corresponding digital set does?} We first give an example of a fractal square with $\#\C(G_F)>\#\C(\D)$, which serves also as an illustration of the level-2 graph.

\begin{exa}\label{exa:last}
Let $F=(F+\D)/10$, where $Q_1, Q_2$ are shown in Figure~\ref{fig:lastexa}.
\begin{figure}[htbp]
\centering \includegraphics[scale=0.6]{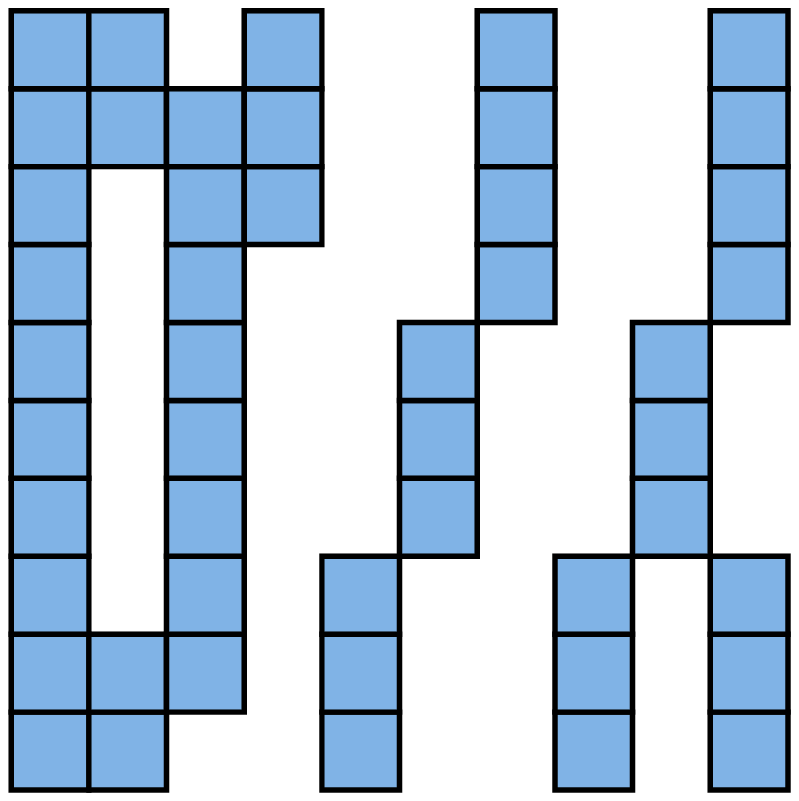}  \quad\quad \includegraphics[scale=0.51]{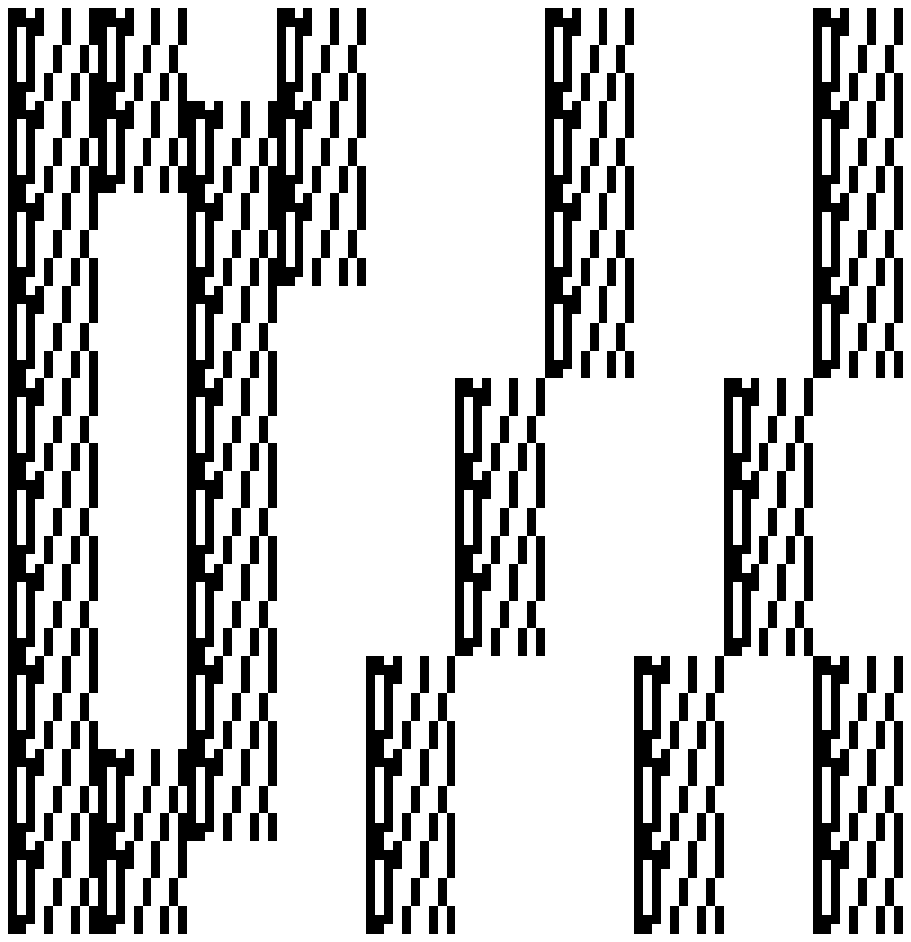}
\caption{$Q_1$ and $Q_2$, where $\#\C(G'_F)=\#\C(G_F)=4$ but $\#\C(\D)=3$.}\label{fig:lastexa} 
\end{figure}
Clearly $\#\C(\D)=3$. Let
\begin{equation*}
\D_1 = \{(i,k) \in \D: 0\leqslant i \leqslant 3\},\,\, \D_2=\{(i,k)\in\D: 4\leqslant i\leqslant 6\},\,\, \D_3=\{(i,k)\in\D: 7\leqslant i\leqslant 9\}. 
\end{equation*}
Note that the leftmost component in $Q_1$ will split into two components in $Q_2$. By definition we can draw $G_F$ and $G'_F$ as in Figure~\ref{fig:lastexag1} and Figure~\ref{fig:lastexag2}. Here $\#\C(G_F) = \#\C(G'_F)=4$ (so by Theorem~\ref{thm:main2}, this fractal square turns out to contain only $4$ connected components).
\begin{figure}[htbp]
\centering \includegraphics[scale=0.26]{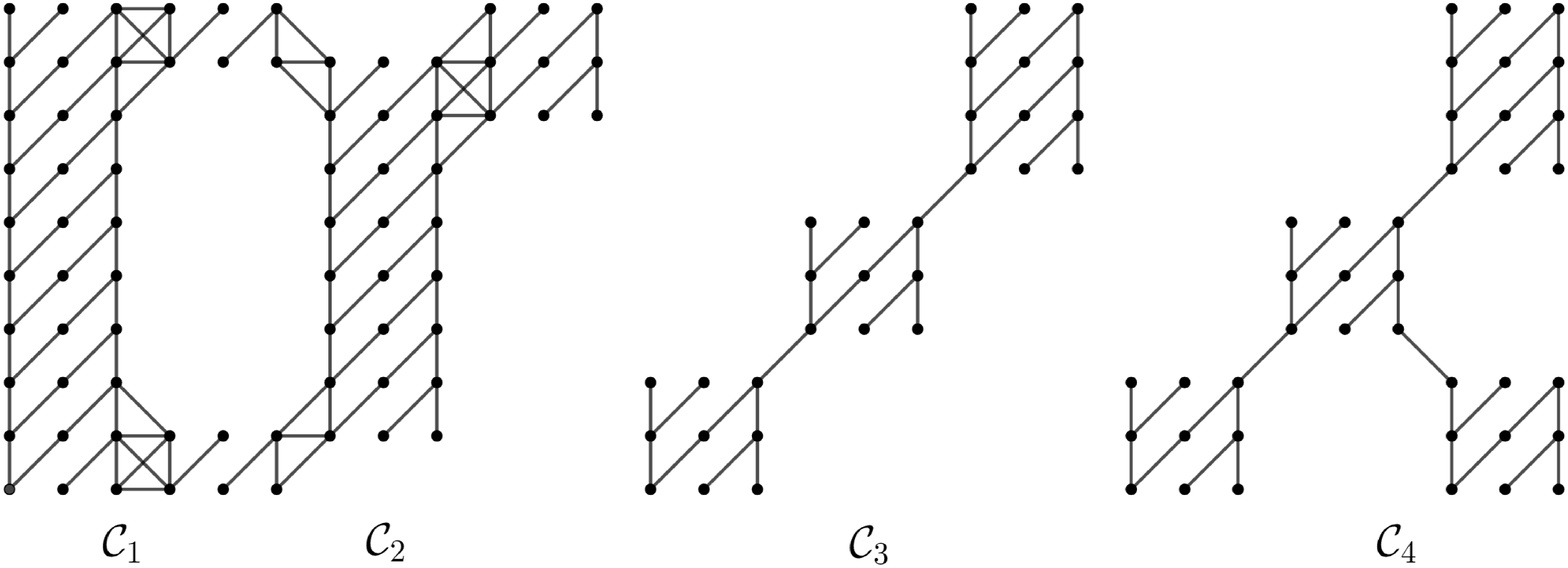}
\vspace{-0.6cm}
\caption{$G_F$ of Example~\ref{exa:last}, where $\#\C(G'_F)=\#\C(G_F)=4$ while $\#\C(\D)=3$.}\label{fig:lastexag1}
\end{figure}
\begin{figure}[htbp]
\centering \includegraphics[scale=0.28]{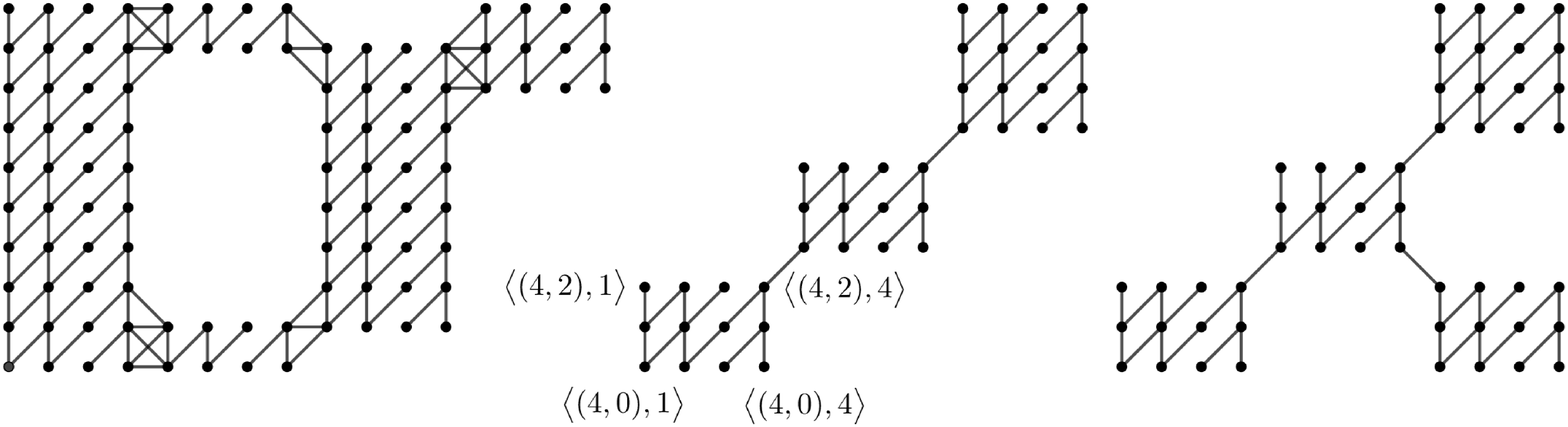}
\vspace{-0.7cm}
\caption{$G'_F$ of Example~\ref{exa:last}, where $\#\C(G'_F)=\#\C(G_F)=4$ while $\#\C(\D)=3$.}\label{fig:lastexag2}
\end{figure}
\end{exa}


Recall that for a disconnected fractal square $F=F(N, \D)$, we always assume $ \C(\D)=\{\D_1, \D_2, \ldots, \D_m\}$ where $m \geqslant 2$. 
\begin{de}
For $1 \leqslant i  \leqslant m$, we call $\D_i$ \emph{vertical-like} if
\begin{equation*} 
(\Z\times\{0\}) \cap \D_i \neq \varnothing \quad \text{and} \quad (\Z\times\{N-1\}) \cap \D_i \neq \varnothing. 
\end{equation*}
$F$ is said to be \emph{vertical-like} if $\D_i$ is vertical-like for all $1 \leqslant i \leqslant m$; similarly we call $\D_i$ \emph{horizontal-like} if
\begin{equation*} 
(\{0\}\times\Z) \cap \D_i \neq \varnothing \quad \text{and} \quad (\{N-1\}\times\Z) \cap \D_i \neq \varnothing, 
\end{equation*}
and $F$ is said to be \emph{horizontal-like} if $\D_i$ is horizontal-like for all $1 \leqslant i \leqslant m$. 
\end{de}
Note that fractal squares in previous examples are all vertical-like. By definition, we immediately have
\begin{fact}\label{fact:1}
  If one of $\D_1,\ldots,\D_m$ is vertical-like then others cannot be horizontal-like.
\end{fact}

The following proposition presents a necessary condition for a fractal square to have only finitely many connected components.
\begin{prop}\label{prop:verhor}
  Let $F=F(N, \D)$ be a disconnected fractal square. If $\#\C(F)<\infty$ then $F$ is either vertical-like or horizontal-like.
\end{prop}
\begin{proof}
Clearly $\#\D>1$. Assume that $F$ is neither vertical-like nor horizontal-like. We first claim that there exists  an $i_0\in\{1,\ldots,m\}$ such that $\D_{i_0}$ is neither vertical-like nor horizontal-like. Since $F$ is not horizontal-like, we know from definition that there exists a $\D_k$ which is not horizontal-like. If $\D_k$ is also not vertical-like then the assertion clearly holds. Otherwise $\D_k$ is vertical-like. By Fact~\ref{fact:1}, $\D_i$ cannot be horizontal-like for each $i \not=k$. Since $F$ is also not vertical-like, there exists an $i_0 \not=k$ such that $\D_{i_0}$ is not vertical-like. Thus $\D_{i_0}$ is neither vertical-like nor horizontal-like.

For this $i_0$, it is not difficult to see that $Q_1|_{\D_{i_0}}$ must be contained in one of the four shaded squares in Figure~\ref{fourcases}.
\begin{figure}[htbp]
\centering\includegraphics[scale=0.15]{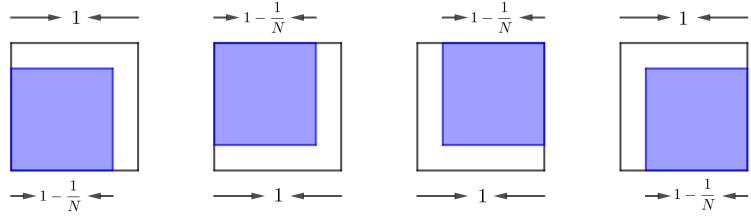}\caption{Four cases: the peripheral square is $[0, 1]^2$}\label{fourcases}
\end{figure}
Then one can conclude that $\#\C(F)=\infty$ (For example, if it were the first case, the bottom square in the leftmost column in $Q_1$ will split and contribute to a new connected component in $Q_2$. This procedure will go on and therefore $\#\C(F)=\infty$).
\end{proof}

\begin{remark}\label{rem:2}
It is noteworthy that for any $n \geqslant 2$, we can always regard $Q_n$ as the first stage in the geometric construction of $F$ since $\bigcap_{k=1}^\infty Q_{kn} = F$. More precisely, if we let
\begin{equation*} 
\D' = N^{n-1}\D+N^{n-2}\D+\cdots+\D, 
\end{equation*}
and set $\psi_d(x)=(x+d)/N^n$ for $d \in \D'$, then the fractal square generated by $\{\psi_d:d\in\D'\}$ coincides with $F$. Further by Proposition~\ref{prop:verhor}, if $\#\C(F)<\infty$, then either all connected components of $\D'$ are vertical-like or all of them are horizontal-like. In particular, for each $\mathcal{B}\in\C(Q_n)$, there exist $0\leqslant a,b\leqslant N^n-1$ such that 
\[ \Big[\frac{a}{N^n}, \frac{a+1}{N^n}\Big] \times \{0\} \subset \B \quad \text{ and } \quad \Big[\frac{b}{N^n}, \frac{b+1}{N^n}\Big] \times \{1\} \subset \B. \]
%
\end{remark}

\begin{coro}\label{cor:cooffver}
Suppose $F$ is a vertical-like fractal square with $\#\C(F)<\infty$, then every connected component $\c$ of $F$ is also ``vertical-like'', i.e., there exist $0\leqslant a,b \leqslant 1$ and a continuous curve $\gamma \subset \c$ joining $(a,0)$ and $(b,1)$.	
\end{coro}
\begin{proof}
Since $F=\bigcap_{n=1}^\infty Q_n$ is the limit set of a decreasing sequence, we can find $\{\B_n\}_{n=1}^\infty$, where $\B_n \in \C(Q_n)$ such that $\B_{n+1} \subset \B_n$ and $\c=\bigcap_{n=1}^\infty \B_n$. In view of the above remark, there exist $a_n,b_n$ such that 
\[ \Big[\frac{a_n}{N^n}, \frac{a_n+1}{N^n}\Big] \times \{0\} \subset \B_n \quad \text{ and } \quad \Big[\frac{b_n}{N^n}, \frac{b_n+1}{N^n}\Big] \times \{1\} \subset \B_n. \]
Since $\{\B_n\}$ is decreasing, we can choose $a_n, b_n$ properly to make above two sequences of closed intervals both decreasing. Then it suffices to choose $a=\bigcap_{n=1}^\infty [\frac{a_n}{N^n},\frac{a_n+1}{N^n}]$ and $b=\bigcap_{n=1}^\infty [\frac{b_n}{N^n},\frac{b_n+1}{N^n}]$. The existence of $\gamma$ follows from the equivalence between connectedness and path connectedness in $F$.	
\end{proof}


In the rest of this section, we always assume that $F=F(N,\D)$ is vertical-like, i.e., $\D_1,\D_2,\ldots,\D_m$ are all vertical-like. Thus for any distinct $i, j \in \{1, \ldots, m\}$, it is not difficult to see that either $\D_i$ lies on the left of $\D_j$ or on the right. More precisely, either
\begin{align*}
&\max\{s: (s, k) \in \D_i\} < \min\{s: (s, k) \in \D_j\}, \quad 0 \leqslant k \leqslant N-1, \text{or} \\
&\max\{s: (s, k) \in \D_j\} < \min\{s: (s, k) \in \D_i\}, \quad 0 \leqslant k \leqslant N-1.
\end{align*}
Rearranging if necessary we may assume
\begin{equation*} 
\max\{s: (s, k) \in \D_i\} < \min\{s: (s, k) \in \D_{i+1}\}, \quad 0 \leqslant k \leqslant N-1 
\end{equation*}
for every $1 \leqslant i \leqslant m-1$, and naturally say that the sequence $\{\D_i\}_{i=1}^m$ is \emph{arranged from left to right}.

\begin{de}
We call $P \subset \D$ a \emph{pillar} if $P=\{(a,b),(a,b+1),\ldots,(a,b+k)\}$ for some integers $0\leqslant a,b \leqslant N-1$ and $k \geqslant 0$, while $(a,b-1), (a,b+k+1) \notin P$.	
\end{de}
 The reason we call $P$ a pillar stems from the observation that $\bigcup_{d \in P}\varphi_d([0,1]^2)$ is a pillar in $Q_1$. For example, for the fractal square in Example~\ref{exa:1}, $\{(0,0),(0,1),(0,2),(0,3),(0,4)\}$ is a pillar. 


The following observation gives us the existence of some specific edges in $G_F$ when $\#\C(F)<\infty$.

\begin{prop}\label{prop:intersect}
Suppose $F=F(N,\D)$ is a vertical-like fractal square with $\#\C(F)<\infty$ and $\#\C(\D)=m>1$. Also $\{\D_i\}_{i=1}^m$ is arranged from left to right. Then
\begin{equation*} 
F_1 \cap \bigl( F_1 + (0, 1) \bigr) \neq \varnothing \quad \text{and} \quad F_m \cap \bigl( F_m+(0, 1) \bigr) \neq \varnothing. 
\end{equation*}
\end{prop}
\begin{proof}
For any $d \in \D$, we temporarily denote $x_d$ to be its first coordinate and $y_d$ to be the second, i.e., $d=(x_d, y_d)$. The proposition is proved by contradiction. Assume that $F_1 \cap \bigl( F_1 + (0, 1) \bigr) = \varnothing$. Equivalently $F_1 \cap \bigl( F_1 - (0, 1) \bigr) = \varnothing$. Let $P \subset \D_1$ be a leftmost pillar, i.e., for any $d \in P$ we have $x_d = \min\{x: (x, y) \in \D_1\}$. Choose $d_0 \in P$ with the largest $2$nd coordinate, i.e., $y_{d_0}=\max\{y_d:d\in P\}$.

We claim that there exists an $i_0 \in \{2,\ldots,m\}$ such that $F_1 \cap (F_{i_0}-(0,1)) \neq \varnothing$. In fact, if $F_1 \cap \bigl( F_k-(0, 1) \bigr) = \varnothing$ for $2 \leqslant k \leqslant m$, then $F_1 \cap (F-(0,1))=\varnothing$. Thus for all $d\in\D\setminus\{d_0\}$ with $x_d=x_{d_0}$, we have $(F_1+d_0)\cap(F+d)=\varnothing$, and therefore $\varphi_{d_0}(F_1) \cap \varphi_d(F)=\varnothing$. Since $P$ is a leftmost pillar, we know that for all $d\in\D$ with $x_d\neq x_{d_0}$, $(F_1+d_0)\cap(F+d)=\varnothing$, so $\varphi_{d_0}(F_1) \cap \varphi_{d}(F)=\varnothing$. In conclusion,
\begin{equation*} 
\varphi_{d_0}(F_1) \cap \varphi_d(F_i) = \varnothing, \quad \forall (d,i)\not=(d_0,1).
\end{equation*}
This means there are some connected components of $F$ contained in $\varphi_{d_0}(F_1) \subset \varphi_{d_0}([0,1]^2)$, which cannot be vertical-like.
From Corollary~\ref{cor:cooffver} we see that  $\#\C(F)=\infty$, which is a contradiction.

Similarly, there exists a $j_0 \in \{2,\ldots,m\}$ such that $F_1 \cap (F_{j_0}+(0, 1)) \neq \varnothing$.

It follows from $F_1 \cap \bigl(F_{i_0}-(0, 1)\bigr) \neq \varnothing$, $F_1 \cap \bigl(F_{j_0}+(0, 1)\bigr) \neq \varnothing$ and \eqref{eq:Fi} that we can find $d_*, d^* \in \D_1$, $d_{i_0} \in \D_{i_0}$ and $d_{j_0} \in \D_{j_0}$, where $y_{d_*}=y_{d_{j_0}}=0, y_{d^*}=y_{d_{i_0}}=N-1$, such that
\begin{equation*} 
\varphi_{d_*}(F) \cap \big(\varphi_{d_{i_0}}(F)-(0,1)\big) \neq \varnothing, \quad \varphi_{d^*}(F) \cap \big( \varphi_{d_{j_0}}(F) + (0, 1) \big) \neq \varnothing. 
\end{equation*}
This implies that $d_{i_0} - d_*, d^* - d_{j_0} \in \{(\epsilon, N-1):\, \epsilon=0,\pm 1\}$. Since $\D_1$ lies on the left of $\D_{i_0}$ and $\D_{j_0}$, we have $x_{d_*} < x_{d_{j_0}}$ and $x_{d^*} < x_{d_{i_0}}$. Thus
\begin{equation}\label{eq:3-3}
 x_{d_{j_0}} - x_{d^*} \geqslant (x_{d_*} +1) - (x_{d_{i_0}}-1) = (x_{d_*}-x_{d_{i_0}})+2.
\end{equation}
Combining this with $x_{d_*}-x_{d_{i_0}},x_{d_{j_0}} - x_{d^*}\in \{0,\pm 1\}$, we have $x_{d_*}-x_{d_{i_0}}=-1$ and $x_{d_{j_0}} - x_{d^*}=1$.
Meanwhile, the inequality in \eqref{eq:3-3} should be an equality so that $x_{d_{j_0}}=x_{d_*} +1$ and $x_{d^*}=x_{d_{i_0}}-1$. Thus $x_{d_*}=x_{d^*}=x_{d_{i_0}}-1=x_{d_{j_0}}-1$. It follows that
\begin{equation*} 
\varphi_{d_*}(F) \cap \big( \varphi_{d_{i_0}}(F) -(0,1) \big) = \varphi_{d_{j_0}}(F)  \cap \big( \varphi_{d^*}(F)-(0,1) \big) 
\end{equation*}
is a singleton. Consequently, 
\[ \varphi_{d_*}((1,0))=\varphi_{d_{j_0}}((0,0)) \in \big(\varphi_{d_*(F)} \cap \varphi_{d_{j_0}}(F)\big). \]
This implies that $d_*$ and $d_{j_0}$ belong to the same connected component of $\D$. But this cannot happen since $\D_1 \neq \D_{j_0}$.

Now we obtain a contradiction and hence $F_1 \cap \bigl( F_1 + (0, 1) \bigr) \neq \varnothing$. Similarly $F_m \cap \bigl( F_m+(0, 1) \bigr) \neq \varnothing$.
\end{proof}

Given any pillar $P \subset \D$, $\{(d,i):d \in P,1\leqslant i \leqslant m\}$  is a subset of the vertex set of $G_F$, and we denoted it by $G_F[P]$ for notational simplicity. 

\begin{lemma}\label{lem:2}
Suppose $\#\C(\D)=m>1$, $P_1, P_2 \subset \D$ are pillars with $\#P_1\leqslant \#P_2$. If $G_F[P_1]$ is connected then so is $G_F[P_2]$.
\end{lemma}
\begin{proof}
Suppose $G_F[P_1]$ is connected. Then $\#P_1>1$ (otherwise $G_F[P_1]$ has $m$ trivial connected components). Without loss of generality, we may assume
\begin{equation*}
P_1 = \{(a, b), (a, b+1), \ldots, (a, b+p)\}, \quad P_2 = \{(c, d), (c, d+1), \ldots, (c, d+q)\},
\end{equation*}
where $0 \leqslant a,b,c,d \leqslant N-1$ and $1 \leqslant p \leqslant q$. If $q=p$, note that $\bigcup_{d \in P_2}\varphi_d(F)$ is just a translation of $\bigcup_{d \in P_1}\varphi_d(F)$, so $G_F[P_2]$ is also connected. Suppose $q>p$ and decompose $P_2$ into
\begin{equation*} 
P_2 = \bigcup_{k=0}^{q-p} \{(c, d+k), (c, d+k+1), \ldots, (c, d+k+p)\} := \bigcup_{k=0}^{q-p} P_{2,k}.
\end{equation*}
Since $\#P_{2, k} = \#P_1 = p$ for each $k$, by the same reason as previous case we know that $\{(d,i):d \in P_{2,k},1 \leqslant i \leqslant m\}$ is connected. Moreover, $P_{2,k} \cap P_{2,k+1} \neq \varnothing$ (since $p>1$) implies that $G_F[P_2]=\{(d,i):d \in \bigcup_{k=0}^{q-p}P_{2,k}, 1\leqslant i\leqslant m\}$ is also connected. 
\end{proof}

\begin{prop}\label{prop:mindiscon}
Suppose $F=F(N, \D)$ is a vertical-like fractal square with $\#\C(F)<\infty$ and $\#\C(\D)=m>1$. Let $P_0$ be a pillar with the least number of elements. If $\#\C(G_F)>m$ then $G_F[P_0]$ is disconnected.
\end{prop}
\begin{proof}
A direct application of Proposition~\ref{prop:intersect} shows that $F \cap \bigl( F+(0, 1) \bigr) \neq \varnothing$. Hence for each pillar $P \subset \D$, there exists a unique $\D_i$ such that $P \subset \D_i$.

We prove the proposition by contradiction. Suppose $G_F[P_0]$ is connected. Then by Proposition~\ref{lem:2}, $G_F[P]$ is connected for every pillar $P$. Fix any $1\leqslant i_0\leqslant m$, we claim that $\{(d,j):d \in \D_{i_0},1\leqslant j\leqslant m\}$ is connected. By the arbitrariness of $i_0$ this implies $\#\C(G_F)=m$ and leads to a contradiction.

Arbitrarily pick $d, d ' \in \D_{i_0}$ and $1 \leqslant j, j' \leqslant m$. By the connectedness of $\D_{i_0}$, we can find a sequence of $\{d_k\}_{k=1}^n \subset \D_{i_0}$ such that $d_1=d$, $d_n = d'$, and $(F+d_k) \cap (F+d_{k+1}) \neq \varnothing$ for each $k$. If we let $P_k$ denote the pillar to which $d_k$ belongs, then 
\begin{equation*} 
\Big( \bigcup_{d \in P_k} \varphi_d(F) \Big) \cap \Big( \bigcup_{d \in P_{k+1}} \varphi_d(F) \Big) \neq \varnothing.
\end{equation*}
Thus for each $1 \leqslant k \leqslant n-1$, there exists an edge joining one vertex in $G_F[P_k]$ and another vertex in $G_F[P_{k+1}]$. Thus their union is connected since each $G_F[P_k]$ is connected. In particular, $(d,j)$ and $(d',j')$ belongs to the same connected component of $G_F$, which proves our assertion.
\end{proof}

\section{Proof of Theorem~\ref{thm:main2}}


In this section, we always assume that $F=F(N, \D)$ is a fractal square with $\#\C(\D)=m>1$ and $\#\C(G_F)=M$, say $\C(G_F) = \{\c_1, \c_2, \cdots, \c_M\}$. 
We first prove Lemma~\ref{lem:3}.


\begin{proof}[Proof of Lemma~\ref{lem:3}]
For any $d \in \D$ and any $1 \leqslant i \leqslant m$, by the connectedness of $\D_i$ we see that $Nd+\D_i$ is a connected subset of $\D^*$. Also note that
\begin{equation}\label{eq:sec3-3}
\varphi_d(F_i) = \frac{F_i+d}{N} = \frac{\frac{F+\D_i}{N}+d}{N} = \frac{F+Nd+\D_i}{N^2} = \bigcup_{d' \in Nd+\D_i} \psi_{d'}(F).
\end{equation}
If $(d, i)$ and $(d', i')$ belong to the same connected component of $G_F$, i.e., there exists a sequence $\{(d_k, i_k)\}_{k=1}^n$ such that $(d_1, i_1) = (d, i)$, $(d_n, i_n) = (d', i')$, and $\varphi_{d_k}(F_{i_k}) \cap \varphi_{d_{k+1}}(F_{i_{k+1}}) \neq \varnothing$ for $1 \leqslant k \leqslant n-1$. By \eqref{eq:sec3-3} this is equivalent to
\begin{equation*}\label{eq:sec3-4}
\Big( \bigcup_{d' \in Nd_k+\D_{i_k}}\psi_{d'}(F) \Big) \cap \Big( \bigcup_{d' \in Nd_{k+1}+\D_{i_{k+1}}}\psi_{d'}(F) \Big) \neq \varnothing, \quad k=1,2,\ldots, n-1.
\end{equation*}
Since $Nd_k+\D_{i_k}$ is a connected subset of $\D^*$ for each $k$, this implies that $\bigcup_{k=1}^n (Nd_k+\D_{i_k})$ is also a connected subset of $\D^*$. In particular, $Nd+\D_i$($=Nd_1+\D_{i_1}$) and $Nd'+\D_{i'}$($=Nd_n+\D_{i_n}$) belong to the same connected component of $\D^*$. In conclusion, $\bigcup_{(d,i) \in \c} (Nd+\D_i)$ is a connected subset of $\D^*$ for any $\c \in \C(G_F)$. Hence $\#\C(\D^*) \leqslant \#\C(G_F)$.

On the other hand, note that for any pair of distinct $\c,\c' \in \C(G_F)$ we have 
\[ \Big( \bigcup_{(d,i) \in \c} \varphi_d(F_i) \Big) \cap \Big( \bigcup_{(d,i) \in \c'} \varphi_d(F_i) \Big) = \varnothing. \]
By \eqref{eq:sec3-3}, this implies
\begin{equation*} 
\Big( \bigcup_{(d,i) \in \c}\bigcup_{d' \in Nd+\D_i} \psi_{d'}(F) \Big) \cap \Big( \bigcup_{(d,i) \in \c'}\bigcup_{d' \in Nd+\D_i} \psi_{d'}(F) \Big) = \varnothing, 
\end{equation*}
and therefore $\bigcup_{(d,i) \in \c} (Nd+\D_i)$ and $\bigcup_{(d,i) \in \c'} (Nd+\D_i)$ must belong to different connected components of $\D^*$. Thus we have $\#\C(\D^*) \geqslant \#\C(G_F)$.

From above arguments, we see that $\#\C(\D^*) = \#\C(G_F) = M$, say $\C(\D^*) = \{\D^*_1, \D^*_2, \ldots, \D^*_M\}$, and rearranging if necessary we have $\D^*_j =  \bigcup_{(d,i) \in \c_j} (Nd+\D_i)$ for $1 \leqslant j \leqslant M$. Combining with \eqref{eq:sec3-3},
\[ \bigcup_{(d,i) \in \c_j} \varphi_d(F_i) = \bigcup_{(d,i)\in\c_j}\bigcup_{d' \in Nd+\D_i} \psi_{d'}(F) = \bigcup_{d' \in \D^*_j} \psi_{d'}(F). \]
\end{proof}


\begin{lemma}\label{lem:3-8}
	$\#\C(F)\geqslant \#\C(G'_F) \geqslant \#\C(G_F)$.
\end{lemma}
\begin{proof}
By the same argument as in the proof of Lemma~\ref{lem:2-2}, we have $\#\C(F)\geqslant \#\C(G'_F)$.

In view of Lemma~\ref{lem:3}, suppose $\D^*=\{\D^*_1, \D^*_2, \ldots, \D^*_M\}$. Also by Corollary~\ref{cor:ejfij}, for any $ 1 \leqslant j \leqslant M$ there exists a unique $1 \leqslant i_j \leqslant m$ such that $F^*_j \subset F_{i_j}$. In particular, $\varphi_d(F^*_j) \subset \varphi_d(F_{i_j})$ for any $d \in \D$. Thus if $\varphi_d(F^*_j) \cap \varphi_{d'}(F^*_{j'})\neq\varnothing$ then we must have $\varphi_d(F_{i_j}) \cap \varphi_{d'}(F_{i_{j'}}) \neq \varnothing$. Combining this fact with definitions of $G_F$ and $G'_F$, we obtain $\#\C(G'_F) \geqslant \#\C(G_F)$.
\end{proof}

In view of Remark~\ref{rem:2}, if $\#\C(F)<\infty$, one can focus on cases when $\D^*_1, \ldots, \D^*_M$ are all vertical-like. Further, we can still arrange $\{\D^*_j\}_{j=1}^M$ from left to right, and in this case we also, for convenience, say that $\{\c_j\}_{j=1}^M$ is \emph{arranged from left to right}.

Similarly, we can say something on the existence of specific edges in $G'_F$ as in Proposition~\ref{prop:intersect}.
\begin{coro}\label{cor:kintersect}
Suppose $F$ is a vertical-like fractal square with $\#\C(F)<\infty$ and $\#\C(\D)=m\geqslant 2$. Assume that $\{\c_j\}_{j=1}^M$ is arranged from left to right, then
\begin{equation*} 
F^*_1 \cap \bigl( F^*_1+(0, 1) \bigr) \neq \varnothing, \quad  F^*_M \cap \bigl( F^*_M+(0, 1) \bigr) \neq \varnothing. 
\end{equation*}
\end{coro}
\begin{proof}
Since the connected components of $\D^*$ are arranged from left to right, we have $F^*_1\subset F_1$ and $F^*_M \subset F_m$. Thus the desired result follows immediately from Proposition~\ref{prop:intersect}.
\end{proof}

The ``if\," part of Theorem~\ref{thm:main2} is a direct result of  the following proposition.

\begin{prop}\label{prop:last1}
If $\#\C(G'_F) = \#\C(G_F)$, then $\#\C(F)=\#\C(G_F)$.
\end{prop}
\begin{proof}
By Lemma~\ref{lem:3}, we again set $\C(\D^*)=\{\D^*_1,\ldots,\D^*_M\}$ with $\D^*_1,\ldots,\D^*_M$ arranged from left to right. From \eqref{eq:ej} and Corollary~\ref{cor:ejfij}, for any $d\in\D$ and $1\leq j\leq M$, there is a unique $k$ such that $\varphi_d(F_j^*) \subset F^*_k$. Since $F^*_1,\ldots,F^*_M$ are disjoint, it follows from $\#\C(G'_F) = \#\C(G_F)$ that 
\begin{equation}\label{eq:cgf}
\C(G'_F) = \big\{\{\langle d,j \rangle: \varphi_d(F^*_j) \subset F^*_1 \}, \ldots, \{\langle d,j \rangle: \varphi_d(F^*_j) \subset F^*_M\} \big\}.
\end{equation}
Denote $\c'_k:=\{\langle d,j \rangle: \varphi_d(F^*_j) \subset F^*_k\}$, $1\leqslant k\leqslant M$. Note that \eqref{eq:cgf} also implies that $F^*_k=\bigcup_{\langle d,j \rangle \in \c'_k} \varphi_d(F^*_j)$.

Set $A_1=\bigcup_{d\in\D^*_j}\psi_d([0,1]^2)$ and recursively define $A_{n+1}:= \bigcup_{d \in \D} \varphi_d(A_n)$, $n \in \Z^+$. Since $A_1=Q_2$, we see that $A_n=Q_{n+1}$ and hence $\bigcap_{n=1}^\infty A_n=F$. We claim that $\#\C(A_n)\leqslant M$ for every $n$, then it follows from Lemma~\ref{lem:Mun} that $\#\C(F)\leqslant M$, and by Lemma~\ref{lem:3-8} we have $\#\C(F)=M=\#\C(G_F)$.

Define $A_1|_j:=\bigcup_{d\in\D_j^*}\psi_d([0,1]^2)$ for $1\leqslant j\leqslant M$. By the connectedness of $\D_j^*$ we know that every $A_1|_j$ is connected. Note that 
\[ A_1 = \bigcup_{d \in \D^*}\psi_d([0,1]^2) = \bigcup_{j=1}^M\bigcup_{d\in\D_j^*}\psi_d([0,1]^2) = \bigcup_{j=1}^M A_1|_j. \] 
This implies that
\begin{align*}
A_2=\bigcup_{d \in \D} \varphi_d(A_1) = \bigcup_{d \in \D}\bigcup_{j=1}^M \varphi_d	(A_1|_j) = \bigcup_{\langle d,j \rangle} \varphi_d(A_1|_j) = \bigcup_{k=1}^M\bigcup_{\langle d,j\rangle \in \c'_k} \varphi_d(A_1|_j).
\end{align*}
Fix any $1\leqslant k\leqslant M$. For every pair of $\langle d,j \rangle, \langle d',j' \rangle \in \c'_k$, by the construction of $G'_F$ we can find $\{\langle d_t,j_t \rangle\}_{t=1}^n \subset \c'_k$ such that $\langle d_1,j_1 \rangle=\langle d,j \rangle$, $\langle d_n,j_n \rangle=\langle d',j' \rangle$, and $\varphi_{d_t}(F^*_{j_t}) \cap \varphi_{d_{t+1}}(F^*_{j_{t+1}}) \neq \varnothing$ for every $1\leqslant t\leqslant n-1$. Also $F^*_{j_t}=\bigcup_{d\in \D^*_{j_t}}\psi_d(F^*) \subset A_1|_{j_t}$, hence 
\[ \varphi_{d_t}(A_1|_{j_t}) \cap \varphi_{d_{t+1}}(A_1|_{j_{t+1}}) \neq \varnothing, \quad \forall 1\leqslant t\leqslant n-1. \] 
Since each $A_1|_{j_t}$ is connected, this implies that $\varphi_d(A_1|_j)$ and $\varphi_{d'}(A_1|_{j'})$ lie in the same connected component of $A_2$. We can conclude now that $\bigcup_{\langle d,j\rangle \in \c'_k} \varphi_d(A_1|_j)$ is connected, and therefore $\#\C(A_2)\leqslant M$.

Let $A_2|_k:=\bigcup_{\langle d,j\rangle \in \c'_k} \varphi_d(A_1|_j)$, $1\leqslant k\leqslant M$. From the above proof, each $A_2|_k$ is a connected set and $A_2=\bigcup_{k=1}^M A_2|_k$. Note that 
\[ A_3=\bigcup_{d\in\D} \varphi_d(A_2) = \bigcup_{d\in\D}\bigcup_{k=1}^M\varphi_d(A_2|_k) =\bigcup_{t=1}^M\bigcup_{\langle d,k\rangle \in \c'_t} \varphi_d(A_2|_k). \]
Also we can deduce from the fact $F^*_j \subset A_1|_j$ that $A_2|_k = \bigcup_{\langle d,j\rangle \in \c'_k} \varphi_d(A_1|_j) \supset \bigcup_{\langle d,j\rangle \in \c'_k} \varphi_d(F^*_j) = F^*_k$ (see the last equality in the beginning of this proof). Applying the same argument we can show that $\#\C(A_3) \leqslant M$. This procedure can go on and we finally obtain that $\#\C(A_n)\leqslant M$ for each $n \in \Z^+$, which is all we need.
\end{proof}




In order to prove the ``only if\," part of Theorem~\ref{thm:main2}, we shall start with introducing a few notations and then give an important observation.
Comparing with the previous notation $G_F[P]$, for any pillar $P \subset \D$, we let $G'_F[P]$ denote $\{\langle d, j\rangle: d \in P, 1 \leqslant j \leqslant M\}$, which is a subset of the vertex set of $G'_F$. Suppose $\{\D_i\}_{i=1}^m$ and $\{\c_j\}_{j=1}^M$ are both arranged from left to right.  By Proposition~\ref{prop:intersect}, vertices $\{(d, 1)\}_{d \in P}$ are connected together in $G_F[P]$, and similarly so are $\{(d, m)\}_{d \in P}$. We always denote the connected components of $G_F[P]$ they belong to by $\c_{P, L}$ and $\c_{P, R}$ respectively. Analogously, using  Corollary~\ref{cor:kintersect}, we denote by $\c'_{P, L}$ (resp. $\c'_{P, R}$) the connected components of $G'_F[P]$ containing $\{\langle d,1 \rangle\}_{d\in P}$ (resp. $\{\langle d,M \rangle\}_{d\in P}$).

For example, for the fractal square $F$ in Example~\ref{exa:last} and the pillar $P=\{(1, 0), (1, 1)\}$, we have (see Figure~\ref{fig:se-s'e} for an illustration)
\begin{equation*}
\begin{gathered}
\c_{P,L} = \big\{\big((1,0),1\big),\big((1,1),1\big),\big((1,1),2\big)\big\},\,\, \c_{P,R} = \big\{ \big((1,0),2\big),\big((1,0),3\big),\big((1,1),3\big) \big\}, \\
\c'_{P,L} = \big\{ \big\langle(1,1),3\big\rangle,\big\langle(1,s),t\big\rangle,s=0,1,t=1,2 \big\},\,\, \c'_{P,R} = \big\{ \big\langle(1,0),3\big\rangle,\big\langle(1,0),4\big\rangle,\big\langle(1,1),4\big\rangle \big\}.
\end{gathered}	
\end{equation*}
\vspace{-0.5cm}
\begin{figure}[htbp]
\centering \includegraphics[scale=0.32]{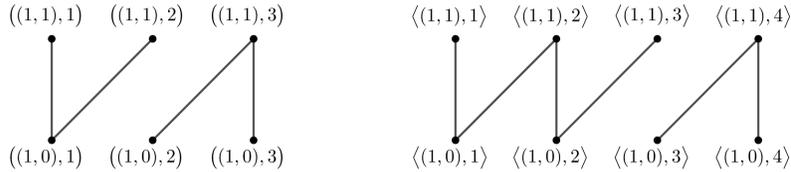}
\vspace{-0.6cm}
\caption{$G_F[P]$ and $G'_F[P]$ for $P=\{(1, 0), (1, 1)\}$ in Example~\ref{exa:last}}\label{fig:se-s'e}
\end{figure}

\begin{lemma}\label{lem:4}
Suppose $F=F(N, \D)$ is a vertical-like fractal square with $\#\C(F)<\infty$ and $\#\C(G_F)>\#\C(\D)>1$. Let $P_0 \subset \D$ be a pillar with the least number of elements. If $\#P_0 < N$ then $\#\C(G_F[P_0]) = \#\C(G'_F[P_0])=2$. Moreover,
\begin{equation*} 
\bigcup_{(d,i) \in \c_{P_0,L}} \varphi_d(F_i) = \bigcup_{\langle d,j \rangle \in \c'_{P_0,L}} \varphi_{d}(F^*_j), \quad \bigcup_{(d,i) \in \c_{P_0,R}} \varphi_d(F_i) = \bigcup_{\langle d,j \rangle \in \c'_{P_0,R}} \varphi_{d}(F^*_j). 
\end{equation*}
\end{lemma}
\begin{proof}
It follows from Proposition~\ref{prop:mindiscon} that $G_F[P_0]$ is disconnected, i.e., $\#\C(G_F[P_0]) \geqslant 2$. If $G_F[P_0]$ contains a connected component other than $\c_{P_0, L}$ and $\c_{P_0, R}$, say $\c$, then
\begin{equation}\label{eq:lem4}
\Big( \bigcup_{(d,i)\in\c} \varphi_d(F_i) \Big) \cap \Big( \bigcup_{(d,i) \in \c_{P_0,L}} \varphi_d(F_i) \Big) = \varnothing, \quad \Big( \bigcup_{(d,i)\in\c} \varphi_d(F_i) \Big) \cap \Big( \bigcup_{(d,i) \in \c_{P_0,R}} \varphi_d(F_i) \Big) = \varnothing.
\end{equation}
Note that for every $d' \notin P_0$,
\[ \varphi_{d'}(F) \cap \Big( \bigcup_{d \in P_0}\varphi_d(F) \Big) = \varphi_{d'}(F) \cap \Big( \bigcup_{d \in P_0} \varphi_d(F_1) \cup \varphi_d(F_m) \Big). \]
Since $\bigcup_{(d,i)\in\c} \varphi_d(F_i) \subset \bigcup_{d \in P_0}\varphi_d(F)$, combining with \eqref{eq:lem4} we know that $\varphi_{d'}(F) \cap \bigcup_{(d,i)\in\c} \varphi_d(F_i) = \varnothing$. Thus there are some connected components of $F$ contained in $\bigcup_{(d,i)\in\c} \varphi_d(F_i) \subset \bigcup_{d \in P_0} \varphi_d([0,1]^2)$.
Since $\#P_0 < N$, these components cannot be vertical-like and therefore $\#\C(F)=\infty$ (by Corollary~\ref{cor:cooffver}), which leads to a contradiction. Thus $G_F[P_0]$ contains exactly two connected components $\c_{P_0, L}$ and $\c_{P_0, R}$.

Note that $\#\C(G'_F[P_0]) \geqslant \#\C(G_F[P_0]) = 2$ (for the same reason as $\#\C(G'_F) \geqslant \#\C(G_F)$). Similarly as above, we can show that $\#\C(G'_F[P_0])=2$, and 
\begin{align*}
\Big( \bigcup_{(d,i) \in \c_{P_0,L}} \varphi_d(F_i) \Big) \cup \Big( \bigcup_{(d,i) \in \c_{P_0,R}} \varphi_d(F_i) \Big) &= \bigcup_{d \in P_0}\varphi_d(F) \\
&= \Big( \bigcup_{\langle d,j \rangle \in \c'_{P_0,L}} \varphi_{d}(F^*_j) \Big) \cup \Big( \bigcup_{\langle d,j \rangle \in \c'_{P_0,R}} \varphi_{d}(F^*_j) \Big),
\end{align*} 
Moreover, note that $F^*_1 \subset F_1, F^*_M \subset F_m$ (since they are both on the leftmost or rightmost), we have
\begin{equation*} 
\bigcup_{\langle d,j \rangle \in \c'_{P_0,L}} \varphi_{d}(F^*_j) \subset \bigcup_{(d,i) \in \c_{P_0,L}} \varphi_d(F_i), \quad \bigcup_{\langle d,j \rangle \in \c'_{P_0,R}} \varphi_{d}(F^*_j) \subset \bigcup_{(d,i) \in \c_{P_0,R}} \varphi_d(F_i),
\end{equation*}
and obtain the desired equality.
\end{proof}


\begin{coro}\label{cor:last}
Under the assumptions of the above lemma, if $\C(G_F[P_0])=\{\c_{P_0, L}, \c_{P_0, R}\}$, then for any other pillar $P$, $G_F[P]$ has at most two connected components $\c_{P, L}$ and $\c_{P, R}$, and
\begin{equation}\label{eq:3-9}
 \bigcup_{(d,i) \in \c_{P,L}} \varphi_d(F_i) = \bigcup_{\langle d,j \rangle \in \c'_{P,L}} \varphi_{d}(F^*_j), \quad \bigcup_{(d,i) \in \c_{P,R}} \varphi_d(F_i) = \bigcup_{\langle d,j \rangle \in \c'_{P,R}} \varphi_{d}(F^*_j). 
\end{equation}
\end{coro}
The proof is very similar to the one of Lemma~\ref{lem:2}.
\begin{proof}
If $\#P=\#P_0$, then $\bigcup_{d \in P}\varphi_d(F)$ is just a translation of $\bigcup_{d \in P_0}\varphi_d(F)$ and the corollary follows immediately from Lemma~\ref{lem:4}. For the case when $\#P > \#P_0$, we may again assume that
\begin{equation*}
P_0 = \{(a, b), (a, b+1), \ldots, (a, b+p)\}, \quad P = \{(c, d), (c, d+1), \ldots, (c, d+q)\},
\end{equation*}
where $0 \leqslant a,b,c,d \leqslant N-1$ and $1 \leqslant p< q$. Denote
\begin{equation*} 
P_k := \{(c, d+k), \ldots, (c, d+k+p)\}, \quad k=0,1,\ldots,q-p. 
\end{equation*}
Clearly $P = \bigcup_{k=0}^{q-p} P_k$. Since $\#P_k=\#P_0$, $G_F[P_k]$ contains exactly two connected components $\c_{P_k, L}$ and $\c_{P_k, R}$ (Here we abuse the notation since $P_k$ is no longer a pillar, but these expressions retain their clarity). Meanwhile, $\c_{P_k, L} \subset \c_{P, L}$ and $\c_{P_k, R} \subset \c_{P, R}$. Note that by $p>1$ we have $\c_{P_k,L} \cap \c_{P_{k+1},L} \neq \varnothing$ and $\c_{P_k,R} \cap \c_{P_{k+1},R} \neq \varnothing$, thus $G_F[P]$ has at most two connected components $\c_{P, L}$ and $\c_{P, R}$. Similarly, $G'_F[P]$ has at most two connected components $\c'_{P, L}$ and $\c'_{P, R}$. Furthermore, by Lemma~\ref{lem:4} we also know 
\[ \bigcup_{(d,i)\in\c_{P_k,L}} \varphi_d(F_i) = \bigcup_{\langle d,j \rangle \in \c'_{P_k,L}} \varphi_{d}(F^*_j), \quad  \bigcup_{(d,i)\in\c_{P_k,R}} \varphi_d(F_i) = \bigcup_{\langle d,j \rangle \in \c'_{P_k,R}} \varphi_{d}(F^*_j), \] 
holds for each $k$, and hence \eqref{eq:3-9} holds.
\end{proof}

The following result is the ``only if\," part of Theorem~\ref{thm:main2}.
\begin{prop}\label{prop:last}
If $F$ is a vertical-like fractal square with $\#\C(F)<\infty$, then $\#\C(G'_F) = \#\C(G_F)$.
\end{prop}
\begin{proof}
In the case when $\#\C(G_F)=\#\C(\D)$, it follows from Theorem~\ref{thm:main1} and Lemma~\ref{lem:3-8} that
\begin{equation*}
\#\C(\D)=\#\C(G_F)\leqslant \#\C(G'_F) \leqslant \#\C(F) = \#\C(\D),
\end{equation*}
so that $\#\C(G'_F)= \#\C(G_F)$. Thus we may assume that $\#\C(G_F)>\#\C(\D)$.

As before we set $\C(\D)=\{\D_1,\ldots,\D_m\}$, $\C(G_F)=\{\c_1,\ldots,\c_M\}$, where $M>m\geqslant 2$, and $\{\D_i\}_{i=1}^m$ and $\{\c_j\}_{j=1}^M$ are both arranged from left to right. Let $P_0 \subset \D$ be a pillar with the least number of  elements. If $\#P_0=N$, $F$ is the product of a Cantor set with $[0,1]$, and hence the union of infinitely many parallel line segments. Thus we may assume $\#P_0<N$.

In order to show $\#\C(G'_F)=M$, it suffices to show that $\{\langle d,j \rangle: \varphi_d(F^*_j) \subset F^*_{j_0}\}$ is a connected component in $G'_F$ for every fixed $1\leqslant j_0\leqslant M$. Recall from Lemma~\ref{lem:3} that $F^*_{j_0} = \bigcup_{(d,i)\in\c_{j_0}}\varphi_d(F_i)$. Note that for any such $\langle d', j' \rangle, \langle d'', j'' \rangle$, by Corollary~\ref{cor:ejfij} and the above equality, there exist $i', i''$ with $(d',i'), (d'', i'') \in \c_{j_0}$, such that $F^*_{j'}\subset F_{i'}$ and $F^*_{j''}\subset F_{i''}$. In particular, we have $\varphi_{d'}(F^*_{j'}) \subset \varphi_{d'}(F_{i'})$ and $\varphi_{d''}(F^*_{j''}) \subset \varphi_{d''}(F_{i''})$. Moreover, since $(d',i'), (d'', i'') \in \c_{j_0}$, we can then find a sequence of $\{(d_k, i_k)\}_{k=1}^n \subset \c_{j_0}$ such that $(d_1, i_1) = (d', i')$, $(d_n, i_n) = (d'', i'')$, and
\begin{equation}\label{eq:sec3-9}
\varphi_{d_k}(F_{i_k}) \cap \varphi_{d_{k+1}}(F_{i_{k+1}}) \neq \varnothing, \quad k=1,2,\ldots,n-1.
\end{equation}
Let $P_k$ denote the pillar to which $d_k$ belongs. By Corollary~\ref{cor:last}, $G_F[P_k]$ contains at most two connected components $\c_{P_k,L}$ and $\c_{P_k,R}$. Thus there exists a sequence of letters $\{t_k\} \subset \{L,R\}$ such that $(d_k,i_k) \in \c_{P_k,t_k}$. By \eqref{eq:sec3-9} this implies
\begin{equation*} 
\Big( \bigcup_{(d,i)\in\c_{P_k,t_k}}\varphi_d(F_i) \Big) \cap \Big( \bigcup_{(d,i)\in\c_{P_{k+1},t_{k+1}}}\varphi_d(F_i) \Big) \neq \varnothing, \quad k=1,2,\ldots,n-1. 
\end{equation*}
Then it follows immediately from \eqref{eq:3-9} that
\begin{equation}\label{eq:last}
\Big( \bigcup_{\langle d,j \rangle \in\c'_{P_k,t_k}}\varphi_d(F^*_j) \Big) \cap \Big( \bigcup_{\langle d,j \rangle \in\c'_{P_{k+1},t_{k+1}}}\varphi_d(F^*_j) \Big) \neq \varnothing, \quad k=1,2,\ldots,n-1. 
\end{equation}
This means $\c'_{P_k, t_k}$ and $\c'_{P_{k+1}, t_{k+1}}$ must lie in the same connected component of $G'_F$. In particular, $\c'_{P_1, t_1}$ and $\c'_{P_n, t_n}$ lie in the same connected component of $G'_F$. Note that we have
\begin{equation*}
\begin{gathered}
\varphi_{d'}(F^*_{j'}) \subset \varphi_{d'}(F_{i'}) \subset \bigcup_{(d,i)\in\c_{P_1,t_1}}\varphi_d(F_i) = \bigcup_{\langle d,j \rangle \in \c'_{P_1,t_1}}\varphi_d(F^*_j), \\
\varphi_{d''}(F^*_{j''}) \subset \varphi_{d''}(F_{i''}) \subset \bigcup_{(d,i)\in\c_{P_n,t_n}}\varphi_d(F_i) = \bigcup_{\langle d,j \rangle \in \c'_{P_n,t_n}}\varphi_d(F^*_j),
\end{gathered}
\end{equation*}
which (by \eqref{eq:last}) implies that $\langle d', j' \rangle$ and $\langle d'', j'' \rangle$ belong to the same connected component of $G'_F$.
\end{proof}

\section{Further remarks}

\begin{remark}
	Theorem~\ref{thm:main1} can be easily extended to higher dimensional cases. In fact, for any fixed integer $N \geqslant 2$ and any non-empty $\D \subset \{0,1,\ldots,N-1\}^n$ where $n \geqslant 2$ is the dimension of the Euclidean space $\R^n$, one can similarly define $F=F(N,\D)$ to be the unique non-empty compact set such that $F=(F+\D)/N$, and call it an \emph{$n$-dimensional fractal cube}. We can analogously apply the previous procedure to this case and reach the same conclusion, i.e., if $\#\C(\D)=m\geqslant 2$, then $\#\C(F)=m$ if and only if $\#\C(G_F)=m$.
\end{remark}

Unfortunately, we do not know whether Theorem~\ref{thm:main2} holds or not in higher dimensional cases. One might also expect to determine that by an induction process, i.e., an $n$-dimensional fractal cube has only finitely many connected components if and only if the projection of it to each face of the unit cube $[0,1]^n$ has only finitely connected components. As we shall see in the following example, however, this turns out not to be the case.


%

\begin{exa}
	Let $\D=\{(0,0,0), (0,1,1),(1,0,1),(1,1,0)\}$ and $F = (F+\D)/2$, then $\pi_{xy}(F) = \pi_{yz}(F) = \pi_{xz}(F) = [0,1]^2$, where $\pi_{xy}, \pi_{yz}, \pi_{xz}$ are projection mappings defined by
\begin{equation*} 
\pi_{xy}: (x,y,z) \mapsto (x,y), \quad \pi_{yz}: (x,y,z) \mapsto (y,z), \quad \pi_{xz}: (x,y,z) \mapsto (x,z). \end{equation*}	
However, $F$ contains infinitely many connected components. In fact, $\#\C(Q_n)=4^{n-1}$.
\end{exa}

On the other hand, given any fractal square $F=F(N,\D)$, we can also consider the cardinality of $\#\C(F)$. In Section 2 we have constructed fractal squares with exactly $m$ connected components for any integer $m \geqslant 2$. It turns out that if $\#\C(F)=\infty$ then the infinity should be uncountable. 

Since every element in $\C(F)$ is a compact subset of $F$, we can endow $\C(F)$ with the Hausdorff metric $h$, i.e.,
\[ h(\c_1,\c_2) = \inf\{\delta>0: \c_1 \subset N(\c_2,\delta) \text{ and } \c_2 \subset N(\c_1,\delta)\}, \quad \c_1,\c_2 \in \C(F), \]
where $N(\cdot,\delta)$ represents the $\delta$-neighborhood. For more details about the Hausdorff metric, please see Falconer~\cite{Fal14}. The following result is well-known (e.g., see Falconer~\cite[Exercise 14.1]{Fal14}).

\begin{lemma}
Every perfect set in a metric space is uncountable.	
\end{lemma}

We say $\c \in\C(F)$ is \emph{of corner type}, if $\c$ is contained in one of the shaded square in Figure~\ref{fourcases}.

\begin{prop}\label{prop:cortype}
Suppose $F$ is not a singleton. If there is a $\c_* \in \C(F)$ which is of corner type, then $\C(F)$ is uncountable. 
\end{prop}
\begin{proof}
We may assume that $\c_* \subset [0,\frac{N-1}{N}]^2$. By the above lemma, it suffices to show that $\C(F)$ has no isolated points.  For any $\c \in \C(F)$ and any $n \in \Z^+$, there exists a connected component of $Q_n$, say $\c_n$, such that $\c \subset \c_n$. Note that $\c_n$ consists of squares with side length $1/N^n$. We can easily select a square such that there are no other squares in $\c_n$ located on the left or the bottom or the bottom left of it. Denote this sqaure by $R_n$, and let $\varphi_n$ be the self-similar mapping which maps $[0, 1]^2$ to $R_n$. From the location of $R_n$ it is easy to see that $\varphi_n(\c_*) \cap (F \setminus R_n)=\varnothing$ and therefore $\varphi_n(\c_*)$ is also a connected component of $F$. 

However, $\varphi_n(\c_*)$ might coincide with $\c$. To fix this problem, we shall find a connected component of $F$ other than $\c_*$ which is also of corner type. This is achieved as follows. Since $\c_*$ is of corner type, we know that $F$ is disconnected (by Corollary~\ref{cor:cooffver} and the fact that $F$ is not a singleton). Then we can find a large $K$ and a square $R'\subset Q_K$ such that there are no other squares in $Q_K$ located on the left of the bottom or the bottom left of it, and $\c_* \cap R'=\varnothing$. Denote $\varphi_{R'}$ to be the similitude mapping $[0,1]^2$ to $R$. Then $\varphi_{R'}(\c_*)$ is a connected component of $F$ which is of corner type, and $\varphi_{R'}(\c_*) \neq \c_*$. Also suppose $\varphi_{R'}(\c_*) \subset [0,\frac{N-1}{N}] \times [\frac{1}{N},1]$ (other three cases are similar).

If $\varphi_n(\c_*)=\c$ for some $n$ then we replace $R_n$ with a square in $Q_n$ such that there are no other squares in $\c_n$ located on the left of the top or the top left of it (if $R_n$ already satisfies this condition then there is no need of replacing it), and replace $\varphi_n(\c_*)$ with $\varphi_n(\varphi_{R'}(\c_*))$. We immediately see that $\varphi_n(\varphi_{R'})$ is a connected component of $F$ and $\varphi_n(\varphi_{R'}(\c_*))\neq\c$.

%
%
%
%

For large $n$ we have $R_n \cap \c \neq \varnothing$. This implies that $h(\c,\varphi_n(\c_*)) \leqslant 2/N^n$, which tends to $0$ as $n \to \infty$. In conclusion, $\c$ is not isolated and the result is proved.  
\end{proof}

In view of the above proposition, we shall now restrict our attention on fractal squares with every connected component vertical-like (recall this concept in Corollary~\ref{cor:cooffver}). In this case, for $\c,\c' \in \C(F)$ we naturally say $\c$ is on the \emph{left} of $\c'$ (or $\c'$ is on the \emph{right} of $\c$) if 
\[ \max\{x: (x,y) \in \c\} < \min\{x: (x,y) \in \c'\}, \quad \forall 0\leqslant y\leqslant 1. \]
Denote $\c_l, \c_r$ to be the leftmost and rightmost connected components of $F$. Applying a similar argument as in the proof of Proposition~\ref{prop:intersect} one can show that
\begin{prop}
Suppose $F$ is a fractal square with every connected component vertical-like. If $\C(F)$ is at most countable then $\c_l \cap (\c_l+(0,1)) \neq \varnothing$ and $\c_r \cap (\c_r+(0,1)) \neq \varnothing$.
\end{prop}

\begin{coro}\label{cor:ls1}
Suppose $F$ is a fractal square with every connected component vertical-like and $\C(F)$ at most countable. Then
\begin{enumerate}[(1)]
\item for any pillar $P \subset \D$, either $\#P=N$ or $\bigcup_{d \in P}\varphi_d(F)$ contains at most two connected components.
\item for any $D \in \C(\D)$, if $P$ is a leftmost (or rightmost) pillar of $D$ then either $\#P=N$ or $\bigcup_{d \in D} \varphi_{d}(F)$ is connected.	
\end{enumerate}
\end{coro}
\begin{proof}
(1) By the above proposition we know that $\{\varphi_d(\c_l):d \in P\}$ lie in the same connected component of $F$, and so do $\{\varphi_d(\c_r):d \in P\}$. If there is another connected component $\c$ of $\bigcup_{d \in P}\varphi_d(F)$, then clearly $\c \cap \bigcup_{d \in \D\setminus P} \varphi_d(F) = \varnothing$, and hence $\c \in \C(F)$. If $\#P<N$ then $\c$ is of corner type, which contradicts the fact that $\C(F)$ is at most countable (see Proposition~\ref{prop:cortype}). In conclusion, either $\#P=N$ or $\bigcup_{d \in P}\varphi_d(F)$ has at most two connected components (one contains $\{\varphi_d(\c_l):d \in P\}$, the other contains $\{\varphi_d(\c_r):d \in P\}$, but these two may coincide with each other). \\
(2) Suppose $P$ is a leftmost pillar of $D$. If $\#P<N$ and $\bigcup_{d \in P}\varphi_d(F)$ is disconnected (i.e., it has at least two connected components), then we know from the above proof that $\bigcup_{d \in P}\varphi_d(F)$ has exactly two connected components, one contains $\bigcup_{d \in P}\varphi_d(\c_l)$, and the other contains $\bigcup_{d \in P}\varphi_d(\c_r)$. Since $P$ is leftmost, the component to which $\bigcup_{d \in P}\varphi_d(\c_l)$ belongs is also a connected component of $F$. It follows from $n<N$ that this component is of corner type and leads to a contradiction. In conclusion, if $\#P<N$ then $\bigcup_{d \in P}\varphi_d(F)$ is connected.
\end{proof}

\begin{coro}\label{cor:ls2}
Suppose $F$ is a fractal square with every connected component vertical-like and $\C(F)$ at most countable. For any pillar $P$ with $\#P<N$, if $\bigcup_{d \in P} \varphi_d(F)$ has at most two connected components, then so does $\bigcup_{d \in P'}\varphi_d(F)$ for any other pillar $P'$ with $\#P' \geqslant \#P$.  
\end{coro}
\begin{proof}
From the above proof one can deduce that $\bigcup_{d \in P}\varphi_d(F)$ has exactly two connected components, one contains $\bigcup_{d \in P}\varphi_d(\c_l)$, the other contains $\bigcup_{d \in P}\varphi_d(\c_r)$, but these two may coincide with each other. We omit the rest part of proof here since it is similar to the one of Corollary~\ref{cor:last}. 	
\end{proof}

To prove Theorem~\ref{thm:main3} it suffices to show the following result.
\begin{theorem}
Suppose $F$ is a fractal square with every connected component vertical-like. If $\C(F)$ is at most countable then $\C(F)$ is a finite set.
\end{theorem}
\begin{proof}
Again choose $P_0$ which is a pillar with the least number of elements. \\
\textbf{Case 1}. $\#P_0=N$. It is easy to see that either $F=[0,1]^2$ or $F$ is the product of a Cantor set with $[0,1]$. Notice that $\C(F)$ is a singleton in the former case, and an uncountable set in the latter. \\
\textbf{Case 2}. $\#P_0<N$. It follows from Corollary~\ref{cor:ls1} that $\bigcup_{d \in P_0}\varphi_d(F)$ contains at most two connected components. Combining with Corollary~\ref{cor:ls2}, this implies that $\bigcup_{d \in P}\varphi_d(F)$ contains at most two connected components for every pillar $P$, and hence $\#\C(F)<\infty$ since there are only finitely many pillars.

\end{proof}

\bigskip

\noindent{\bf Acknowledgements.}
The author would like to thank Professor Huo-Jun Ruan for many valuable conversations and suggestions on the final exposition. He also thanks Professor Yang Wang for his helpful advice and Professor Jun-Jie Miao for a discussion on Theorem~\ref{thm:main3}.

\bibliographystyle{amsplain}

\begin{thebibliography}{10}

\bibitem{BanKe}
C. Bandt and K. Keller, \textit{Self-similar sets. II. A simple approach to the topological structure of fractals}, Math. Nachr. {\bf 154} (1991), 27-39.

\bibitem{CriSte}
L. L. Cristea and B. Steinsky, \textit{Connected generalised Sierpi\'nski carpets}, Topology Appl {\bf 157} (2010), 1157-1162.

\bibitem{Fal14}
K. J. Falconer, \textit{Fractal geometry: Mathematical foundations and applications}, 3rd ed., John Wiley \& Sons, Ltd., Chichester, 2014.

\bibitem{FM92}
K. J. Falconer, D. Marsh, \textit{On the Lipschitz equivalence of Cantor sets}, Mathematika \textbf{39} (1992), 223–233.

\bibitem{Hata}
M. Hata, \textit{On the structure of self-similar sets}, Japan J. Appl. Math. {\bf 2} (1985), 381-414.

\bibitem{Hut81}
J. E. Hutchinson, \textit{Fractals and self-similarity}, Indiana Univ. Math. J. {\bf 30} (1981), 713-747.

\bibitem{Kigami}
J. Kigami, \textit{Analysis on Fractals}, Cambridge Tracts in Mathematics. no.143., Cambridge University Press, 2001.

\bibitem{LauLuoRao}
K. S. Lau, J. J. Luo and H. Rao, \textit{Topological structure of fractal squares}, Math. Proc. Camb. Phil. Soc. {\bf 155} (2013), 73-86.

\bibitem{LMR19}
Z. Liang, J. J. Miao and H. J. Ruan, \textit{Topology and gap sequences of fractal squares and Bedford-McMullen carpets}, preprint. 

\bibitem{LuoRaoTan}
J. Luo, H. Rao and B. Tan, \textit{Topological structure of self-similar sets}. Fractals {\bf 10} (2002), 223-227.

\bibitem{LuoWang}
J. J. Luo and L. Wang, \textit{Topological properties of self-similar fractals with one parameter}, J. Math. Anal. Appl. {\bf 457} (2018), 396-409.

\bibitem{Mun}
J. R. Munkres, \textit{Topology}, 2nd ed., Prentice Hall, 2000.

\bibitem{RRW12}
H. Rao, H.-J. Ruan, Y. Wang, \textit{Lipschitz equivalence of Cantor sets and algebraic properties of contraction ratios}, Trans. Amer. Math. Soc. \textbf{364} (2012), 1109–1126.

\bibitem{RRW13}
H. Rao, H.-J. Ruan, Y. Wang, \textit{Lipschitz equivalence of self-similar sets: algebraic and geometric properties}, Contemp. Math. \textbf{600} (2013), 349–364.

\bibitem{RoinM}
K. A. Roinestad, \textit{Geometry of Self-Similar Sets}, MS. thesis, Virginia Polytechnic Institute and State University (2007).

\bibitem{RoinP}
K. A. Roinestad, \textit{Geometry of Fractal Squares}, PhD. thesis, Virginia Polytechnic Institute and State University (2010).

\bibitem{RuanWang}
H. J. Ruan and Y. Wang, \textit{Topological invariants and Lipschitz equivalence of fractal squares}, J. Math. Anal. Appl {\bf 451} (2017), 327-344.


\end{thebibliography}

\end{document}